\documentclass[12pt]{amsart}
\usepackage[utf8]{inputenc}
\usepackage{graphicx}
\usepackage{xcolor}
\usepackage{times}
\usepackage[hidelinks]{hyperref}
\usepackage{enumerate,latexsym}
\usepackage{latexsym}
\usepackage{amssymb}
\usepackage{float}
\usepackage{enumitem}

\usepackage{comment}
\usepackage{amsfonts,amssymb}
\usepackage{dsfont}

\makeatletter
\def\namedlabel#1#2{\begingroup
 #2%
 \def\@currentlabel{#2}%
 \phantomsection\label{#1}\endgroup
}
\makeatother

\newtheorem{theorem}{Theorem}[section]
\newtheorem*{acknowledgement*}{Acknowledgements}

\newtheorem{corollary}[theorem]{Corollary}

\newtheorem{lemma}[theorem]{Lemma}

\theoremstyle{definition}

\theoremstyle{definition}

\newcommand{\ben}{\begin{enumerate}}
\newcommand{\een}{\end{enumerate}}

\newcommand{\ed}{\end{document}}

\definecolor{rrr}{rgb}{.9,0,.1}

\definecolor{rr}{rgb}{.8,0,.3}
\definecolor{pp}{rgb}{.5,0,.7}

\usepackage{pdfsync}
\title{Composition Properties of Hyperbolic Links in Handlebodies}
\author[C. Adams]{Colin Adams}
\address{Department of Mathematics, Williams College, Williamstown, MA 01267}
\email{cadams@williams.edu}

\author[D. Santiago]{Daniel Santiago}
\address{Department of Mathematics, M.I.T., Cambridge, MA 02139}
\email{dsantiag@mit.edu}

\subjclass[2020]{57K10,57K32}
\keywords{Hyperbolic Links, Handlebodies}

\begin{document}  

\begin{abstract} We consider knots and links in handlebodies that have hyperbolic complements and operations akin to composition. Cutting the complements of two such open along separating twice-punctured disks such that each of the four resulting handlebodies has positive genus, and  gluing a pair of pieces together along the twice-punctured disks in their boundaries, we show the result is also hyperbolic. This should be contrasted with composition of any pair of knots in the 3-sphere, which is never hyperbolic. Similar results are obtained when both twice-punctured disks are in the same handlebody and we glue a resultant piece to itself along copies of the twice-punctured disks on its boundary.  We include applications to staked links. %{\color{blue} Just a first attempt. Feel free to edit.}
\end{abstract}

\maketitle

%Color coding: Dora - {\color{magenta}magenta}, Daniel - {\color{red}red}, Zach - {\color{violet}violet}, Alex - {\color{orange}orange}, Max - {\color{teal}teal}, Maya - {\color{green}green}, Joye - {\color{purple}purple}, Ben - {\color{pink}pink}

%{\color{blue} All figures must be pdf, not png}
%{\color{magenta} figures are still drafts/placeholders}
%\\
%{\color{blue} I would suggest a notation like $(H, L)$ for a link in a handlebody and $(H, L, D)$ for a choice of disk in $(H, L)$. This is what Alex Simons and I did when were talking about composition of virtual knots. Helps to keep track of everything at once.}

\section{Introduction}
\label{Intro}
\noindent A compact orientable $3$-manifold $M$ is tg-hyperbolic if the manifold $M'$ obtained from $M$ by shaving off all torus boundaries and capping off all sphere boundaries with balls admits a finite volume hyperbolic metric such that all remaining boundary components are totally geodesic.  For a link $L$ in a handlebody $H$, we say that the pair $(H,L)$ is \emph{tg-hyperbolic} if the complement of an open regular neighborhood of $L$ in $H$ is tg-hyperbolic. By the Mostow-Prasad Rigidity Theorem, such a hyperbolic metric will only depend on the complement $H \setminus L$ up to homeomorphism, which allows us to associate a hyperbolic volume to $(H,L)$ that is invariant under ambient isotopies of $L$ in $H$. 

Work of W. Thurston implies that a compact 3-manifold is tg-hyperbolic if and only if it contains no properly embedded essential disks, spheres, annuli or tori. A surface is essential if  it is not boundary-parallel and it is incompressible and boundary-incompressible. 

Examples of knots and links in handlebodies with complements that are tg-hyperbolic appear in \cite{Adamsnew}, \cite{BakerHoffman},  \cite{Frigerio}, \cite{FMP}, and \cite{simplesmallknots}.  In  \cite{AltPaper}, a large source of such examples is provided.  Results from \cite{hp17} can also be used to generate many more.

Let $L_1$ and $L_2$ be two links in handlebodies $H_1$ of genus $g_1$  and $H_2$ of genus $g_2$ respectively. Just as we have composition of two links in the 3-sphere, we would like to define composition of these links in handlebodies. 

To that end, let $D_1\subset H_1,D_2 \subset H_2$ be properly embedded disks twice punctured by $L_1,L_2$ respectively which separate  balls $B_1$ and $B_2$ from $H_1$ and $H_2$ such that $B_1 \cap L_1$ and $B_2 \cap L_2$ are unknotted arcs. Discarding the balls yields %The triples $(H_1,L_1,D_1)$ and $(H_2,L_2,D_2)$ and an orientation preserving homeomorphism $\phi: D_1 \to D_2$ sending $\partial D_1$ to $\partial D_2$ and $L_1\cap D_1$ to $L_2 \cap D_2$ determine a link $L_3$  in the handlebody $H_3$ of genus $g_1 + g_2$, denoted $(H_1,L_1,D_1) \oplus_{\phi} (H_2,L_2,D_2)$,  as follows. 
%Cutting $H_1$ and $H_2$  along $D_1$ and $D_2$ yields the balls %{\color{blue} Need to say something about discarding the balls that result so you are left with the pieces you describe. Are you allowing knotted arcs in the balls? }{\color{red} No I am not allowing this. Will change.}
%$B_1 \subset H_1, B_2 \subset H_2$ and
two handlebodies $H_1' \subset H_1$ and $ H_2' \subset H_2$. Let $L'_1 = H_1' \cap L_1$ and $L'_2 = H_2' \cap L_2$. Glue $H_1'$ to $H_2'$ along $D_1$ and $D_2$ via $\phi$. Since $\phi$ sends the endpoints of the arc in $L'_1$ to the endpoints of the arc in $L'_2$, this results in a link in a handlebody, denoted $(H'_1,L'_1,D_1) \oplus_{\phi} (H'_2,L'_2,D_2)$ in $H_3$ as in Figure \ref{CompositionFigure}. %{\color{blue} Maybe redefine this to be $(H'_1,L'_1,D_1) \oplus_{\phi} (H'_2,L'_2,D_2)$ to be consistent with later use?} {\color{red} Do you mean $L'_1,L'_2$ are the arcs after cutting? }%{\color{blue} Does this work okay for you?} {\color{red} Yep just need to change the arc notation in the  figure. (Done)}
%If $L,L'$ are knots, this yield an embedding of the composition of $L,L'$ in $S^3$ into $H_3$, as shown in Figure \ref{CompositionFigure}.

\begin{figure}[htbp]
    \centering
    \includegraphics[scale=0.4]{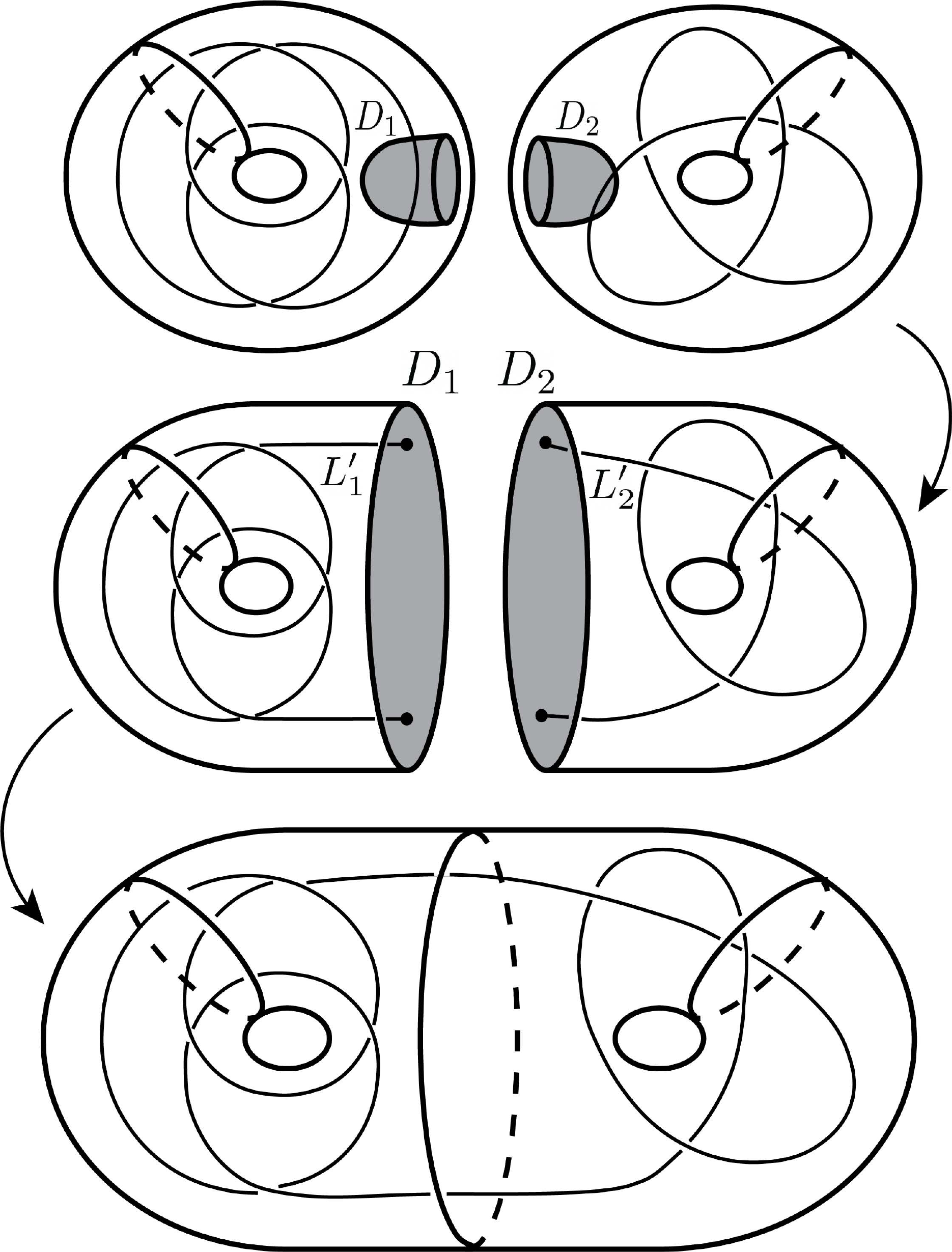}
    \caption{Forming the link $(H'_1,L'_1,D_1) \oplus_{\phi} (H'_2,L'_2,D_2)$}.
    \label{CompositionFigure}
\end{figure}

In contrast to the usual composition of links, the link/handlebody pair  $(H'_1,L'_1,D_1) \oplus_{\phi} (H'_2,L'_2,D_2)$ depends highly on $D_1,D_2,$ and $\phi$. Furthermore, while composition of links in $S^3$ never results in a hyperbolic link, $(H'_1,L'_1,D_1) \oplus_{\phi} (H'_2,L'_2,D_2)$ can be tg-hyperbolic.

However, even if both $H_1 \setminus L_1$ and $H_2 \setminus L_2$ are tg-hyperbolic, it is not always true that $(H'_1,L'_1,D_1) \oplus_{\phi} (H'_2,L'_2,D_2)$ is tg-hyperbolic. In fact, the disks $D_1$ and $D_2$ can always be chosen so that at least one is ``knotted" and  there is an essential torus in the link complement associated to $(H'_1,L'_1,D_1) \oplus_{\phi} (H'_2,L'_2,D_2)$ as shown in Figure \ref{TorusinComp}.

\begin{figure}[htbp]
\includegraphics[scale=0.4]{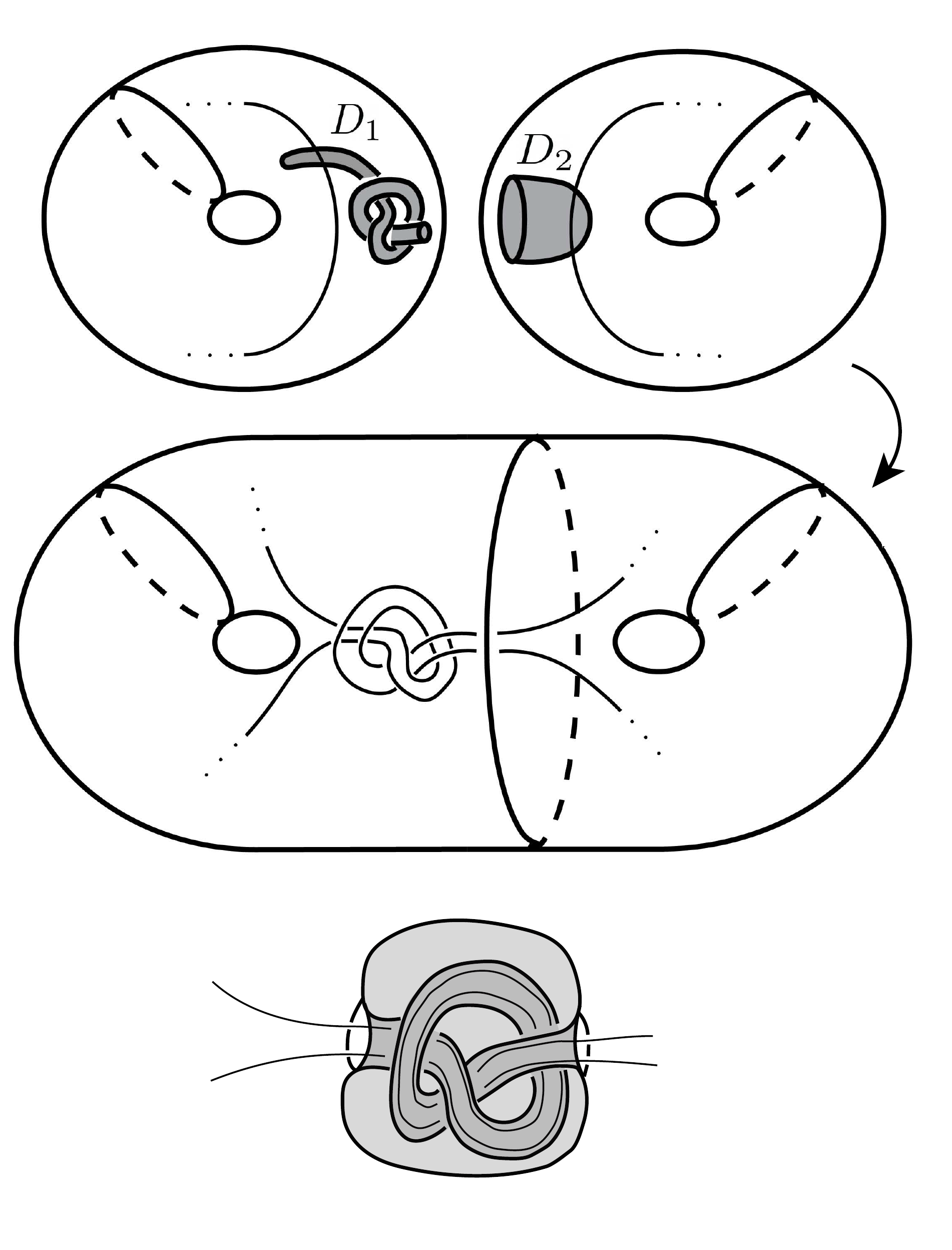}
\caption{By choosing one of $D_1,D_2$ to be ``knotted", one can create an essential torus in the complement  $H_3 \setminus L_3$ which separates a knot exterior from $H_3$ of the form appearing in the last image.} 
%{\color{blue} slight rewording? ``as shown in the last image" might be ``the knot exterior appearing in the last image."}
%This will yield an incompressible torus if the components of $L,L'$ which puncture $D_1,D_2$ respectively are nontrivial in $\pi_1(H_g),\pi_1(H_{g'})$ or are nontrivially linked with other components of $L,L'$ which are nontrivial in $\pi_1(H_g),\pi_1(H_{g'})$, one of which will always hold if $L,L'$ are tg-Hyperbolic. {\color{blue} Fix caption.}

%{\color{blue} I think you do not want your knots to be longitudes in each of the solid tori. I would add in a disk with a $K_1$ and a $K_2$ in it and then the longitude going into it to make it clear this is a very general construction.  And I would draw the knotted torus slightly different, shading the outside of it with an opaque grey. I would also move the one meridian right to the edge where the torus turns inside out and put another one on the opposite end where it turns inside out.}
\label{TorusinComp}
\end{figure}

In Section \ref{SwitchSection}, we provide a method to avoid the problem with ``knotted disks''. In Theorem \ref{SwitchMovetheorem}, we prove that if the two handlebody/link pairs cut along their disks appear as submanifolds of handlebody/link pairs of higher genus that are tg-hyperbolic,  then the composition of the original pair is tg-hyperbolic.
%{\color{blue} Daniel, check that this paraphrase is correct. I find the following version just difficult to parse.} {\color{red} Yep it looks correct.}
%That is to say, we prove that if all genera following are at least one, and one of $H_1 \setminus L_1$ or $H_2 \setminus L_2$ is tg-hyperbolic and there exist triples $(H_{3},L_3,D_3),$ $(H_{4},L_4,D_4)$ and  homeomorphisms $\kappa: D_1 \to D_3, \rho: D_2 \to D_4$ such that $(H_1,L_1,D_1) \oplus_{\kappa} (H_{3},L_3,D_3),(H_{2},L_2,D_2) \oplus_{\rho} (H_{4},L_4,D_4)$ are tg-hyperbolic, then $(H_{1},L_1,D_1) \oplus_{\phi} (H_{2}, L_2,D_2)$ is tg-hyperbolic.
%{\color{blue} Fix notation here.} {\color{red} Should be fixed now}
We also show an analogue of this result where one cuts along two separating twice punctured disks in a single handlebody and glues the resulting manifold to itself along a homeomorphism of the twice punctured disks. 

In Section \ref{AppSection}, we discuss applications. As mentioned,  \cite{AltPaper} and \cite{hp17} provide many examples of tg-hyperbolic links in handlebodies, and our construction here can be applied to them to generate many more.  Furthermore,  these results can be applied to staked links introduced in \cite{generalizedknotoids}, which correspond to link projections with isolated poles placed in the complementary regions, over which strands of the link cannot pass. These are equivalent to links in handlebodies.

We can also consider applications to knotoids. In \cite{hypknotoids}, a definition of what it means for a planar knotoid to be hyperbolic is given in terms of a corresponding knot in a handlebody being tg-hyperbolic. So the results here can be applied to extend the known examples of hyperbolic planar knotoids. %{\color{blue} Added this.}

In addition to considering knots in handlebodies, there is work that has been done on hyperbolicity of links in thickened surfaces, as in \cite{small18} and \cite{hp17}. Questions about compositions have been addressed in that situation, as in \cite{AdamsSimons}. %Although we will not prove it here, 
Converting a method applied there to our situation can avoid the problem of knotted disks and allow composition of tg-hyperbolic links in handlebodies to be tg-hyperbolic without requiring them to be submanifolds as described above. That is, we can take a geodesic $g$ that runs from the surface of the handlebody to the link. Then a regular neighborhood of $g$, including its endpoint on the link, will be a properly embedded twice-punctured disk that cannot be knotted and therefore allows composition to yield tg-hyperbolic links in handlebodies. However, we do not include the details of the proof here.

%In Section \ref{InfFamilySection}, we show that a certain infinite family of knots in the genus 2 handlebody is hyperbolic. This is relevant in other work by the authors, as it shows that a corresponding infinite family of planar knotoids is hyperbolic under the map $\phi^G_{\mathbb{R}^2}$ defined in \cite{hypknotoids}. 

\subsection{Acknowledgements} A special thanks to the other members of the knot theory research group in the SMALL REU program at Williams College in summer of 2022, including Alexandra Bonat, Maya Chande, Maxwell Jiang, Zachary Romrell,Benjamin Shapiro and Dora Woodruff. Without the many helpful conversations with them, this paper would not have come into being.  
 
\section{Proof of Main Result}
\label{SwitchSection}

\noindent 

\noindent For any submanifold $S$ of a smooth manifold $M$, we denote by $N(S)$ a \textit{closed} regular neighborhood of $S$ in $M$ and by $\mathring{N}(S)$ the interior of $N(S)$. For a space $X$, we denote by $|X|$ the number of connected components of $X$.

Let $H_1,H_2$ be two handlebodies %of genus $g_{1,1}+g_{1,2},g_{2,1}+g_{2,2}$ respectively, where $g_{i,j} > 0$, 
that contain links $L_1$ and $L_2$ such that $H_1 \setminus L_1$ and $H_2 \setminus L_2$ are tg-hyperbolic. 
Let $E_1$ and $E_2$ be properly embedded disks in $H_1$ and $H_2$, which separate $H_1$ and $H_2$ into handlebodies $H_{1,1},H_{1,2}$ and $H_{2,1},H_{2,2}$ of genera $g_{1,1},g_{1,2}$ and $g_{2,1},g_{2,2}$ respectively, where all genera are at least 1.  Suppose further that $E_1$ and $E_2$ are each  twice punctured by $L_1$ and $L_2$ respectively. Let $L_{i,j} = L_i \cap H_{i,j}$. 

 We denote by $M_{i,j} = H_{i,j} \setminus \mathring{N}(L_{i,j})$ 
%the submanifolds of the complements $H_i \setminus \mathring{N}(L_{H_i})$ separated by $E_1,E_2$,
and by  $F_i = E_i \setminus \mathring{N}(L_i)$ the corresponding separating surfaces.  %Connecting the endpoints of each $L_{i,j}$  with an embedded arc in $E_i$, and pushing the resulting curve into the interior of $H_{i,j}$ yields a handlebody/link pair  $(H_{i,j},\overline{L}_{i,j})$. 
As we will ultimately only be interested in %$L_{1,1},L_{2,2},
$M_{1,1}$ and $M_{2,2}$, we will for convenience often drop the extra subscripts and write %$L_{1,1},L_{2,2}$,
$M_{1,1}$ and $M_{2,2}$ as %$L_1,L_2$,
$M_1$ and $M_2$ respectively. %{\color{blue} We don't want to have two $L_1$s and two $L_2$s, so I eliminated the option of using $L_1$ and $L_2$ to represent $L_{1,1}$ and $L_{2,2}$. I don't think we use it anyway.}

%{\color{blue} These choices of g, g', $g_1$ and $g_2$ seem arbitrary and confusing. I would start with two handlebodies $H_1$ and $H_2$ and then talk about disks that split the first into handlebodies $H_{1,1}$ of genus $g_{1,1}$, $H_{1,2}$ of genus $g_{1,2}$, and the second into $H_{2,1}$ of genus $g_{2,1}$ and $H_{2,2}$ of genus $g_{2,2}$. But I do understand why you made the choice you did, since you will only consider one pair of the resulting handlebodies.So up to you. Also using notation of pairs and triples may make this less confusing.} 
%{\color{red} Am adding in these changes}%

%{\color{blue} Have to push connection into the interior.}

\begin{figure}[htbp]
\includegraphics[scale=0.6]{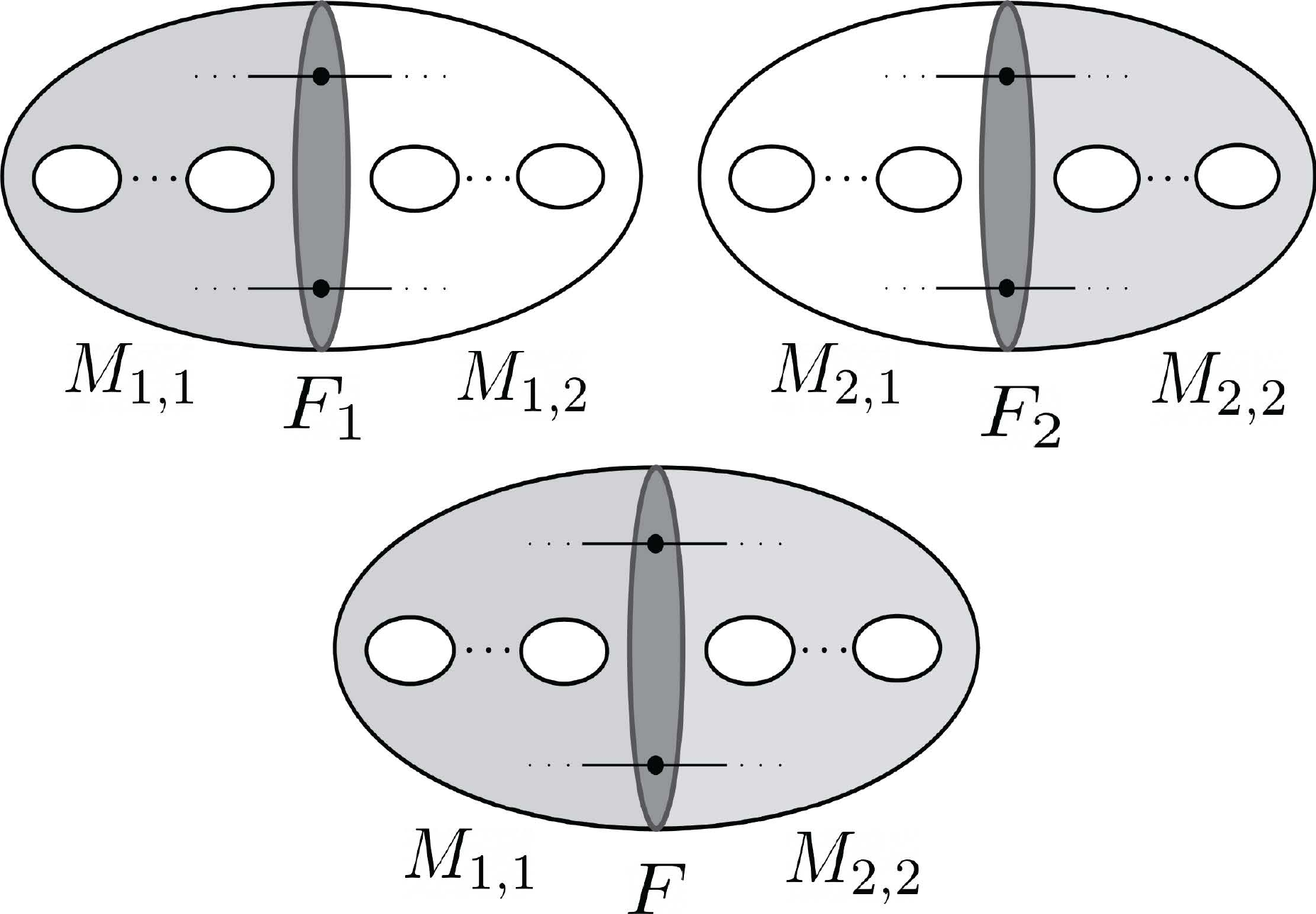}

\caption{The links $L_1,L_2$ in $H_1,H_2$ respectively and the link $L$ in the handlebody $H$.}
\label{SwitchMove}
\end{figure}
\vspace{0.5 cm}

Gluing $H_{1,1}$ to $H_{2,2}$ along an orientation preserving homeomorphism\\ $\phi: F_1 \to F_2$ sending $\partial E_1$ to $\partial E_2$ and $\partial F_1 \cap \partial N(L_{1,1})$ to $\partial F_2 \cap \partial N(L_{2,2})$ yields a manifold/link pair  denoted $(H_{1,1}, L_{1,1}, E_1) \oplus_{\phi} (H_{2,2},L_{2,2},E_2)$ 
which is a handlebody $H$ of  genus $g_{1,1}+g_{2,2}$ containing the link $L$ formed by gluing $L_{1,1}$ and $L_{2,2}$ along their endpoints, as in Figure \ref{SwitchMove}. Let $H_{L}$ be the complement of  $\mathring{N}(L)$ in $H$.
We denote by $F$ the image of $F_1$ and $F_2$ in $H_L$, and by $E$ the separating disk in $H$ corresponding to $F$. 
%We denote by $K_{1,i},K_{2,j}$ the components of $L$ that lie in $H_{1,1},H_{2,2}$ respectively. {\color{blue} Fix this.} These are components of $L_{1,1},L_{2,2}$ respectively. 

There is one component of $L$, which we denote by $K$, that is cut into two arcs $K_1$ and $K_2$ by $F$, the arcs of which are in $H_{1,1}$ and $H_{2,2}$ respectively. %These two arcs correspond to two components $\overline{K}_1$ and $\overline{K}_2$ in $\overline{L}_{1,1}$ and  $\overline{L}_{2,2}$ respectively.   
We denote $\partial N(K)$ by $T_{K}$. We denote  the sub-annuli of $T_{K}$ corresponding to the arcs $K_1$ and $K_2$ by $A_{K_1}$ and $A_{K_2}$ respectively.   A hyperbolic manifold is always assumed tg-hyperbolic unless otherwise stated.

\begin{theorem} \label{SwitchMovetheorem}
Let $L_1$ and $L_2$ be links in $H_1$ and $H_2$ such that $H_1 \setminus L_1$ and $H_2 \setminus L_2$ are tg-hyperbolic, with $E_1 \subset H_1$ and $E_2 \subset H_2$ twice-punctured disks separating each of $H_1$ and $H_2$ into handlebodies, all of positive genus.  If $\phi: E_1 \to E_2$ is a homeomorphism sending $\partial E_1$ to $\partial E_2$ and sending punctures to punctures, then $(H, L) = (H_{1,1}, L_{1,1}, E_1) \oplus_{\phi} (H_{2,2},L_{2,2},E_2)$ is tg-hyperbolic.

%{\color{blue} Reworded to not assume hyperbolicity of $H_{1,1} \setminus \overline{L}_{1,1}$.} {\color{red} Since in the definition of $\oplus_{\phi}$ we require the disks to separate balls containing unknotted arcs, we technically would need to write this as $(H, L) = (H_{1,1}, \overline{L}_{1,1}, D_1) \oplus_{\phi} (H_{2,2},\overline{L}_{2,2},D_2)$, but I think we can just add a comment about this in the setup.}

\end{theorem}

%\noindent Theorem \ref{SwitchMove} implies the result in Section \ref{Intro}{\color{blue} Need the statement in Section 1 to be a theorem so you can refer to it. Not clear here what you are referring to.} by letting $(H_1,L_{H_1}) = (H_g,L,D_1) \oplus_{\kappa} (H_{g''},L'',D_3), (H_2,L_{H_2}) = (H_{g'},L',D_2) \oplus_{\rho} (H_{g'''},L''',D_4)$, so that $(H_g,L) = (H_{1,1},\overline{L}_{1,1}),(H_{g'},L') = (H_{2,2},\overline{L}_{2,2})$ and $(H_{g},L,D_1)\oplus_{\phi} (H_{g'},L',D_2) = (H_L,L)$.

\vspace{0.5 cm}

%\noindent Note that Theorem \ref{SwitchMovetheorem} is equivalent to the formulation given in the introduction, where $(H_{1,1},\overline{L}_{1,1})$ and $(H_{2,2},\overline{L}_{2,2})$ play the roles of $H_1 \setminus L_1$ and $H_2 \setminus L_2$ there. %{\color{blue} Check this.} {\color{red} Looks good}
%From now on, without loss of generality, we assume that $(H_{1,1},\overline{L}_{1,1})$ is tg-hyperbolic. 

To prove Theorem \ref{SwitchMovetheorem}, it is enough to show that since  $H_1  \setminus \mathring{N}(L_1)$ and $H_2\setminus  \mathring{N}(L_2)$
%H_{1,1}\setminus \mathring{N}(\overline{L}_{1,1}), H_1\setminus \mathring{N}(L_1)$ and $H_2 \setminus \mathring{N}(L_2)$ 
contain no essential disks, spheres, annuli and tori, the same holds for $H \setminus \mathring{N}(L)$. 
%{\color{blue} Need to have explained that tg-hyperbolic is equivalent to no essential disks, spheres , annuli or tori in the intro. I put it in. Also , always make your statement if... then... rather than then... if...}
iIn the remainder of this section, we rule out these four kinds of essential surfaces with a sequence of lemmas. %We assume throughout this section that $(H_{1,1},\overline{L}_{1,1})$ is tg-hyperbolic.

%{\color{blue} Should all subsequent $M_1$ and $M_2$ be changed to $M_{1,1}$ and $M_{2,2}$? Or should we say that $M_1$ is the copy of $M_{1,1}$ in $H_L$ and $M_2$ is the copy of $M_{2,2}$ in $H_L$?}

\begin{lemma} The surfaces $F,F_1,F_2$ are incompressible and boundary incompressible in $H_L$.
\label{SurfaceLemmaSwitchMove}
\end{lemma}

\begin{proof}
We show that $F$ is incompressible and boundary incompressible. The same reasoning immediately applies to $F_1$ and $F_2$, as we only use that $M_1$ and $M_2$ are submanifolds of the hyperbolic manifolds $H_1\setminus \mathring{N}(L_1)$ and $H_2 \setminus \mathring{N}(L_2)$ respectively.

Suppose that $F$ is compressible. Then there is some nontrivial circle $C \subset F$ which bounds a disk $D'$ in $M_1$ or $M_2$. Suppose $D' \subset M_1$ and let $D$ be the disk in $E$ bounded by $C$. Suppose that $D$ is punctured once by $L$. Then the sphere $D \cup D'$ (after pushing so intersections are transverse) is punctured once by $K_1$, a contradiction. Suppose next that $D$ is punctured twice by $L$. Then $K_1$ is contained in the $3$-ball bounded by $D \cup D'$ in $H_{1}$, so $K_1$ can be pushed into a neighborhood of  $E$ by an isotopy fixing the endpoints of $K_1$, and hence $M_1$ contains an essential disk with boundary in $\partial H$ isotopic to $\partial E$, which contradicts that $H_1 \setminus \mathring{N}(L_1)$ is hyperbolic. We reach the analogous contradictions if $D' \subset M_2$, since $H_2 \setminus \mathring{N}(L_2)$ is hyperbolic.

\vspace{0.5 cm}

\noindent Suppose next that $F$ is boundary compressible. Then there is a nontrivial arc $\alpha \subset F$ which together with an arc $\beta \subset \partial H_L$ bounds a disk $D$ in $M_1$ or $M_2$  such that $D \cap F = \alpha$. Suppose $D \subset M_1$. There are two cases.

\vspace{0.5 cm}

\noindent \newline \textbf{Case 1}: The arc $\beta$ is in $A_{K_1}$. If $\beta$ is trivial in $A_{K_1}$, then we can isotope $D$ so that $\partial D \subset F$, yielding a compression disk for $F$, a contradiction. If $\beta$ is nontrivial in $A_{K_1}$, then $K_1$ together with an arc in $F$ bounds a disk in $M_1$.  Thus we can push $K_1$ onto $F$ in $H_L$ through an isotopy fixing the endpoints of $K_1$, which implies that $M_1$ contains an essential disk with boundary isotopic in $\partial H_L$ to $\partial E$, which contradicts that $H_1 \setminus \mathring{N}(L_1)$ is tg-hyperbolic.

\vspace{0.5 cm} 

\noindent \newline \textbf{Case 2:} The arc $\beta$ is in $\partial H$. Suppose $D$ is separating in $H_{1,1}$. Then $D$ separates $M_1$ into two regions, each of which contains an endpoint of $K_1$. Since $K_1$ is connected, this is a contradiction. 

Suppose $D$ is not separating in $H_{1,1}$. The arc $\alpha$ separates an annulus $A$ from $F$ such that $A^* = A \cup D$ is a properly embedded annulus in $H_L$ with one boundary component a meridian on $T_K$ and another boundary component on $\partial H$. Since $D$ is not separating in $H_{1,1}$, $\partial A^* \cap \partial H$ is nontrivial in $\partial H$, thus $A^*$ is an essential annulus in $M_1$, which contradicts that $H_1 \setminus \mathring{N}(L_1)$ is hyperbolic.

\vspace{0.5 cm} \noindent Since $H_2 \setminus \mathring{N}(L_2)$ is hyperbolic, we reach the analogous contradictions if $D \subset M_2$, and thus $F$ is boundary incompressible.
\end{proof}

\begin{lemma}
The manifold $H_L$ is irreducible.
\label{NoSpheresSwitchMove}
\end{lemma}

\begin{proof}
Suppose $H_L$ contains an essential sphere $S$. Suppose first that \\ $S \cap F = \emptyset$.  Then $S \subset M_1$ or $S \subset M_2$, which implies that one of  $H_1 \setminus \mathring{N}(L_1)$ or $H_2 \setminus \mathring{N}(L_2)$  contains an essential sphere, a contradiction. 

Suppose next that $S \cap F \neq \emptyset$. We assume that $|S \cap F|$ is minimal among all essential spheres in $H_L$. An innermost circle $C$ of $S \cap F$ in $S$ bounds a disk $D$ in $S$ such that $D \cap F = C$. Since $F$ is incompressible, $C$ bounds a disk $D'$ in $F$. Then we can view $D \cup D'$ as a sphere in $H_{1,1}$ or $H_{2,2}$, which from the last case must bound a ball in $H_L$. Thus we can push $D'$ through $F$ to reduce $|S \cap F|$, contradicting minimality.

\end{proof}

\begin{lemma}
The manifold $H_L$ is boundary irreducible.
\label{NoDisksSwitchMove}
\end{lemma}

\begin{proof}

Suppose $\partial H_L$ has a compressing disk $D'$. Suppose first that $\partial D' \subset \partial  N(L)$. Since a handlebody lives in $S^3$, we need only consider the case that the boundary of the disk is a longitude of  %If $ \partial D'$ is a $(1,0)$ curve in some torus component of $\partial N(L)$, then $D' \cup N$ is a sphere in $H$ punctured once by a component of $L$, where $N$ is the disk in $N(L)$ bounded by $\partial D'$, a contradiction.
some torus component of $N(L)$. Then the corresponding component of $L$ bounds a disk in $H$ which does not intersect a distinct component of $L$. Taking the boundary of a regular neighborhood of the disk union the component implies that $H_L$ contains an essential sphere, which we have already eliminated. %{\color{blue} Why only worry about (0,1) and (1,0) curves? Explain.}
\vspace{0.5 cm}

Suppose now that $\partial D' \subset \partial H$. If $D' \cap F = \emptyset$, then one of $\partial H_{1,1}$ or $\partial H_{2,2}$ has a compression disk in $M_1$ or $M_2$ respectively, which contradicts that $H_1 \setminus \mathring{N}(L_1)$ and $H_2 \setminus \mathring{N}(L_2)$ are hyperbolic. Thus we can assume that $D' \cap F \neq \emptyset$, and we further assume that $|D' \cap F|$ is minimal among all compression disks of $\partial H_L$. Then by incompressibility of $F$, the elements of $A \cap F$ are all arcs. By minimality of $|D' \cap F|$, an outermost arc of $D' \cap F$ in $D'$ is then nontrivial in $F$, as otherwise by doing a surgery we could find a compression disk $D''$ of $\partial H_L$ with $|D'' \cap F| < |D' \cap F|$. This gives a boundary compression for $F$, a contradiction.
\end{proof}

\begin{lemma} The manifold $H_L$ does not contain an essential annulus $A$ with $A \cap F = \emptyset$.
\label{NoAnnulionOneSideSwitchMove}
\end{lemma}
\begin{proof} Suppose $H_L$ contains such an annulus, and assume without loss of generality that $A \subset M_1$. %The analogous contradictions are reached if $A \subset M_2$, since we  will not use the fact here that $H_{1,1} \setminus \overline{L}_{1,1}$ is tg-hyperbolic. %{\color{blue} This is confusing as we are assuming $H_{1,1} \setminus \overline{L_{1,1}}$ is hyperbolic, so that favors the $M_1$ side. In this sense, the two sides are not equivalent.} {\color{red} That is true, but it is also not used in most of the arguments, so I think we could just try to be a bit more clear while still keeping this convention. } {\color{blue} Added sentence above. Does that work?} {\color{red} Yep I think its good.} 
We can view $A$ as a properly embedded annulus $\overline{A}$ in $H_1 \setminus \mathring{N}(L_1)$ which we will show is essential $H_1 \setminus \mathring{N}(L_1)$, a contradiction to its being tg-hyperbolic. 

Suppose $\overline{A}$ is compressible in $H_1 \setminus \mathring{N}(L_1)$. Then a nontrivial simple closed curve $\gamma \subset \overline{A}$ bounds a disk $D$ in $H_1 \setminus \mathring{N}(L_1)$. We assume that $|D \cap F_1|$ is minimal among all compression disks of $\overline{A}$ in $H_1 \setminus \mathring{N}(L_{1})$. Note that the components of $D \cap F_1$ are circles. If $D \cap F_1 = \emptyset$, then $D \subset M_1$, which implies that $A$ is compressible in $H_L$, a contradiction. If $D \cap F_1 \neq \emptyset$, by incompressibility of $F_1$, an innermost circle of $D \cap F_1$ in $D$ is trivial in $F_1$, hence by irreducibility of $H_L$, we can reduce $|D \cap F_1|$ by an isotopy, contradicting minimality.

Thus $\overline{A}$ is boundary compressible in $H_1 \setminus \mathring{N}(L_1)$. (Note that if $\overline{A}$ is boundary parallel, then it is boundary compressible.) Therefore, both boundary components of $\overline{A}$ must be on the same 
%There are two cases.
%\vspace{0.5 cm}
%\noindent \newline \textbf{Case 1:} The annulus $A$ has its boundary components $\partial_1 A,\partial_2 A$ on distinct components of $\partial H_L$.  Then $\overline{A}$ is boundary incompressible and not boundary parallel. 
%\vspace{0.5 cm}
%\noindent \newline \textbf{Case 2:} The annulus $A$ has its boundary components on the same
component of $\partial H_L$. We consider two cases.
%By the same reasoning as in Case 1, $\overline{A}$ is incompressible in $H_1 \setminus \mathring{N}(L_1)$. There are two subcases.

\vspace{0.5 cm}

\noindent \newline \textbf{Case 1:}  The annulus $A$ has both boundary components on $\partial H$. Suppose $\overline{A}$ is boundary compressible in $H_1 \setminus \mathring{N}(L_1)$. Then a  nontrivial arc  in $\overline{A}$ together with an arc in $\partial H_1$ bounds a disk $D$ in $H_1 \setminus \mathring{N}(L_1)$ . We assume $|D \cap F_1|$ is minimal among all boundary compressing disks of $\overline{A}$. If $D \cap F_1 = \emptyset$, then $D \subset M_1$, which implies that $A$ is boundary compressible in $H_L$, a contradiction. If $|D \cap F_1|  \neq \emptyset$, by incompressibility of $F_1$ and minimality, the components of $D \cap F_1$ are arcs. An outermost arc in $D$ must be nontrivial in $F_1$, as otherwise, we could find a boundary compression disk $D'$ of $\overline{A}$ in  $H_1 \setminus \mathring{N}(L_{1})$ with $|D' \cap F_1| < |D \cap F_1|$, a contradiction. But then we have a boundary compression disk for $F_1$ in $H_L$, a contradiction to Lemma \ref{SurfaceLemmaSwitchMove}.

%If $\overline{A}$ were boundary parallel in $H_1 \setminus \mathring{N}(L_{1})$, it would be boundary compressible in $H_1 \setminus \mathring{N}(L_1)$, hence from the last case $\overline{A}$ is an essential annulus in  $H_1 \setminus \mathring{N}(L_1)$, a contradiction to it being tg-hyperbolic.
\vspace{0.5 cm}

\noindent \newline \textbf{Case 2:}  The annulus $A$ has both boundary components on $\partial N(L)$. Suppose first that the components $\partial A$ are on a single torus component of $\partial N(L)$ in $M_1$, and that $\overline{A}$ is boundary compressible in $H_1 \setminus \mathring{N}(L_1)$. A  nontrivial arc  in $\overline{A}$ together with an arc in $\partial N(L_1)$ bounds a disk $D$ in $H_1 \setminus \mathring N(L_1)$ . Note that the components of $D \cap F_1$ are circles, thus repeating the minimality argument from Case $1$ it follows that $\overline{A}$ is boundary compressible in $H_L$, a contradiction. 

Suppose next that the components of $\partial A$ are both in $T_K$. Since $A \cap F = \emptyset$, both components of $\partial A$ are $(1,0)$ curves in $T_K$. Suppose $\overline{A}$ is boundary compressible in $H_1 \setminus \mathring{N}(L_1)$, then a nontrivial arc $\alpha$ in $\overline{A}$ together with an arc $\beta \subset T_{K}$ bounds a disk $D$ in $H_1 \setminus \mathring{N}(L_1)$. Again, choose $D$ such that $|D \cap F_1|$ is minimal. 

If $\beta \cap F = \emptyset$, the components of $|D \cap F_1|$ are circles, and thus we reach a contradiction by repeating the minimality argument from Case $1$ and obtaining a boundary compression for $A$ in $H_L$. If $\beta \cap F \neq \emptyset$, then $\beta$ intersects $\partial N(L_1) \cap M_{1,2}$ in at least one arc. Thus, $D$ must intersect $F$ in at least one arc. Choosing an outermost arc on $D$, we obtain a disk in $D \cap M_{1,2}$ with a boundary consisting of two arcs, one a nontrivial arc in $F$ and one in $\partial N(L_1) \cap M_{1,2}$. This contradicts boundary incompressibility of $F_1$. 

%, contradicting minimality. So the arc is nontrivial on $\partial N(K_2)$. Then $\partial (A \cup D \cup N(K))$ contains an essential  sphere that contradicts our elimination of them previously.
%{\color{blue} Rewrote above.} {\color{red} Looks good, I just made some slight notation corrections but this argument is cleaner.}{\color{blue} Rewrote this lemma since incompressibility applied to both cases so moved that up front.}
%hence $H_L$ contains an essential sphere which bounds a $3$-ball containing $K$, a contradiction. 
%{\color{red} I agree that you do get an essential sphere. The disk $D$ lets you contract $A \cup N(K)$ so it does not interact with the other components. I don't see why there is an essential sphere within that boundary specifically. Also need to be careful with notation, this is happening now completely in $H_1 \setminus \mathring{N}(L_1)$, which is assumed hyperbolic, while $K_2$ lives in $H_2\setminus \mathring{N}(L_2)$. We mean the part of $L_1$ on right side of $F_1$ in $H_1$, which I don't think there is a specific notation for. }
\end{proof}

\begin{lemma}
The manifold $H_L$ contains no essential annuli.
\label{NoAnnuliSwitchMove}
\end{lemma}

\begin{proof}
Suppose $H_L$ contains an essential annulus $A$. We assume that $|A \cap F|$ is minimal among all essential annuli in $H_L$. From Lemma \ref{NoAnnulionOneSideSwitchMove}, we can assume that $A \cap F \neq \emptyset$. There are three cases.
\vspace{0.5 cm}
\noindent \newline \textbf{Case 1:} The annulus $A$ has boundary components  $\partial_1 A,\partial_2 A$ in $\partial H$. By minimality and incompressibility and boundary incompressibility of $F$, the components of $A \cap F$ are all either nontrivial circles in $A$ and $F$ or all nontrivial arcs in $A$ and $F$.
\vspace{0.5 cm}

\noindent \newline (1a) The components of $A \cap F$ are all nontrivial circles in $A$ and $F$. Then up to isotopy, the boundary components $\partial_1 A, \partial_2 A$ do not intersect $F$.  Suppose some component of $\partial A$, say $\partial_1 A$, is in $M_1$. Then a circle $C$ in $A \cap F$  together with $\partial_1 A$ bounds an annulus $A^* \subset A$ in $M_1$   such that $A^* \cap F = C$. 

Let $D$ denote the disk in $E$ bounded by $C$. Suppose $D$ is punctured once by $L$. Then $H_L$ contains a properly embedded once-punctured disk $D \cup A^*$ which can be pushed off $E$ to yield an essential annulus in $M_1$, contradicting Lemma \ref{NoAnnulionOneSideSwitchMove}. 

Suppose $D$ is punctured twice by $L$. Then we can slide $C$ along $E$ out to $\partial H$. Hence we obtain an annulus $A^{**}$ that is entirely contained in $M_1$.

%{\color{blue} $H_{1,1} \setminus \overline{L}_{1,1}?$}

\begin{comment} {\color{blue} Can't we just push $C$ out to $\partial H_1$ and get contradiction to $H_1 \setminus L_1$ being tg-hyperbolic? Don't need complement of $\overline{L}_{i,i} hyperbolic.$}Then we can push the endpoints of $K_1$ off of $D$ into $H_{1,1}$ and connect them to obtain $\overline{L}_{1,1}$. Since $H_{1,1} \setminus \mathring{N}(\overline{L}_{1,1})$ is hyperbolic, {\color{blue} using hyperbolicity of $H_{1,1} \setminus \mathring{N}(\overline{L}_{1,1})$ right here. If necessary, looks like we would need same for $H_{2,2} \setminus \mathring{N}(\overline{L}_{2,2})$.;}this implies that $\partial (A \cup D) = \partial_1 A$ is isotopic in $\partial H$ to $\partial E$, as otherwise $\partial H_{1,1}$ has a compression disk in $H_{1,1} \setminus \mathring{N}(\overline{L}_1)$. Thus $\partial_1 A$ and $C$ together bound an annulus $A'$ in $\partial H_{1,1}$. 
\end{comment}

So,  $A^{**}$ is a properly embedded annulus in $M_1$, which is incompressible since $\partial_1 A, C$ are nontrivial in $A$. Hence by Lemma \ref{NoAnnulionOneSideSwitchMove}, it is boundary compressible in $H_L$ and both boundary curves are on $\partial H_{1}$. 

Doing the boundary compression on $A^{**}$ yields a disk with boundary on $\partial H_{1}$. If the boundary of the disk is trivial on $H_1$, as happens when the two boundaries of $A^{**}$ are parallel on $\partial H_{1}$, then we can form a sphere from the disk and another disk on $\partial H_1$.  Irreducibility of $H_L$ implies we can then isotope $A$ to lower the number of intersections with $F$, a contradiction. 

If the boundary of the disk is nontrivial on $H_1$, we contradict boundary irreducibility of $H_L$.

\vspace{0.5 cm } 

\noindent \newline (1b) The components of $A \cap F$ are nontrivial arcs in both $A$ and $F$. Then $A$ is cut by $F$ into disks in $M_1$ and $M_2$ with boundaries that consist of  two opposite sides in $F$ and two opposite sides in $\partial H$. Let $D_1 \subset M_1$ be one such disk. Let $R \subset F$ be a rectangle such that two opposite sides of $R$ are the components of $D_1 \cap F$, and the other two sides are disjoint curves in $\partial E$. Then $D_1 \cup R$ is either a properly embedded M\"obius band $Q$ or a  properly embedded annulus $A_1 \subset M_1$ in $H_L$. 

We begin with the case it is an annulus, which we claim is essential in $H_L$.  By minimality of  $|A\cap F|$, $A_1$ is incompressible, as otherwise we could push $D_1$ through $F$. 

Suppose $A_1$ is boundary compressible in $H_L$. Then a nontrivial arc $\alpha \subset A_1$ bounds a disk $D$ in $H_L$ with an arc $\beta \subset \partial H$. We suppose $|D \cap F|$ is minimal among all boundary compression disks of $A_1$ in $H_L$. By minimality and incompressibility of $F$, the components of $D \cap F$ are arcs. Up to isotopy we can assume that $\alpha \subset D_1$ or $\alpha \subset R$. In the former case $D$ provides a boundary compression of $A$, a contradiction. Suppose now that $\alpha \subset F$. If $D$ does not intersect $F$ in an arc distinct from $\alpha$, then $D$ provides a boundary compression of $F$, a contradiction. If $D \cap F \neq \emptyset$, then an outermost arc in $D$ of $D \cap F$ is nontrivial in $F$, as otherwise by doing a surgery we could find a boundary compression disk $D'$ of $A_1$ along $\alpha$ with $|D' \cap F| < |D \cap F|$. This yields a boundary compression of $F$, a contradiction. If $A_1$ were boundary parallel in $H_L$, it would be boundary compressible, hence $A_1$ is an essential annulus in $H_L$ contained in $M_1$, which contradicts Lemma \ref{NoAnnulionOneSideSwitchMove}.

Suppose now that $D_1 \cup R$ is a M\"obius band $Q$. Then the boundary of a regular neighborhood of $Q$ is an annulus $A_2$. It cannot compress in the regular neighborhood of $Q$ since that is a solid torus, and the boundaries of $A_2$ are isotopic to twice the core curve of the solid torus. It cannot compress to the outside of the regular neighborhood of $Q$ because either component of the boundary of the annulus links the core curve of the annulus, due to the twisting of the M\"obius band. If the core curve bounded a disk, that disk would not intersect the boundary curves of the annulus, which would contradict the linking. 
  %{\color{blue} Check this.} 
And it is boundary incompressible for the same reasons that $A_1$ is, also contradicting Lemma \ref{NoAnnulionOneSideSwitchMove}. %{\color{blue} Check this.}
%{\color{red} I think I mostly get it, a bit hard for me to visualize I will think about it more.}{\color{red}  Have thought about this and am pretty convinced.}

\vspace{0.5 cm}
\noindent \newline \textbf{Case 2:} The annulus $A$ has boundary components $\partial_1 A$ and $ \partial_2 A$ on $\partial N(L)$. There are two subcases.
\vspace{0.5 cm}
\noindent \newline (2a) Both $\partial_1 A$ and $\partial_2 A$ lie on the torus components $T_{K_{1,i}}$ and $T_{K_{2,j}}$  where $T_{K_{1,i}}$ is a torus component of $\partial N(L)$ contained completely in $M_1$, and  $T_{K_{2,j}}$ is a torus component of $\partial N(L)$ contained completely in $M_2$. By minimality of $|A \cap F|$ and incompressibility of $F$, the components of $A \cap F$ are circles which are nontrivial in both $A$ and $F$. A circle $C$ in $A \cap F$ bounds a subannulus $A^*$ of $A$  with $\partial_1 A$ such that $A^* \cap F = C$ which is incompressible since $C$ and $\partial_1 A$ are nontrivial in $A$. 

Suppose $A^* \subset M_1$.
Let $D$ denote the disk in $E$ bounded by $C$. If $D$ is punctured once, we can take the union of it with $A^*$, and then $H_L$ contains an essential annulus in $M_1$ with one boundary component on $T_{K_{1,i}}$  and another boundary component on $T_{K}$.  If $D$ is punctured twice, we can glue the annulus $F \setminus D$ to $A^*$ to obtain an annulus essential in $H_L$ and contained in $M_1$ with one boundary component on $T_{K_{1,i}}$ and the other boundary component on $\partial H$. Both cases contradict Lemma \ref{NoAnnulionOneSideSwitchMove}. We reach the analogous contradictions if $A^* \subset M_2$. %{\color{blue} Added some detail here.}{\color{red}  Looks good.}
\vspace{0.5 cm}

\noindent \newline (2b) The annulus $A$ has at least one boundary component $\partial_1 A$ on $K$. Suppose first that $\partial_1 A$ is a $(1,0)$ curve in $T_{K}$. Then $\partial_2 A$ is either a $(1,0)$ curve in $T_{K}$ or lies in some $T_{K_{1,i}}$ or $T_{K_{2,j}}$.  By minimality of $|A \cap F|$ and incompressibility of $F$, the components of $A \cap F$ are circles which are nontrivial in $A$ and $F$. A circle $C$ in $A \cap F$ bounds a subannulus $A^*$ of $A$  with $\partial_1 A$ such that $A^* \cap F = C$. Note $A^*$ is incompressible since $C$ and $\partial_1 A$ are nontrivial in $A$. 

Suppose, without loss of generality, that $A^* \subset M_1$. Let $D$ denote the disk in $E$ bounded by $C$. Suppose first that $D$ is punctured once. Then we obtain a new annulus $A'^*$ by gluing $D$ onto $A^*$, with both boundaries now meridians on $T_K$. 
We can view $A'^*$ as a properly embedded annulus in $M_1$ which is boundary compressible in $H_L$ by Lemma \ref{NoAnnulionOneSideSwitchMove}. 

By irreducibility of $H_L$, the annulus must be boundary parallel. If it is boundary parallel to the $M_1$ side of $H_L$, then we can use that to isotope $A$ along $T_K$ and reduce its number of intersection curves with $F$, a contradiction to minimality. It cannot be boundary parallel to the other side as the boundary of the handlebody is to that side.

%{\color{blue} Rewrote the above.}{\color{red} Looks good.}

%A boundary compressing arc of $A'^*$ is either contained in $A_{K_{1}}$ or intersects $A_{K_{2}}$ in a nontrivial arc. The second case implies that $H_L$ has an essential sphere containing $T_{K}$, a contradiction.{\color{blue} Why?} Thus, the first case must hold, and we can push $A^*$ through $F$, lowering the number of intersection curves, and contradicting minimality. 

If $D$ is punctured twice, then $H_L$ contains an essential annulus in $M_1$ with one boundary component on $T_{K}$ and the other boundary component on $\partial H$. this contradicts Lemma \ref{NoAnnulionOneSideSwitchMove}. \\

 Suppose next that $\partial_1 A$ is a $(p,q)$-curve in $T_{K}$ with $|q| > 0$. If $\partial_2 A \subset T_{K}$, then all components of $A \cap F$ are nontrivial arcs in $A$. If there is an  innermost arc of $A \cap F$ in $F$ that is trivial in $F$, then $A$ is boundary compressible, contradicting its essentiality.

So all arcs in $A \cap F$ are nontrivial and parallel on $F$. Each component of $A \cap M_1$ is a disk with boundary consisting of four arcs, two in $\partial N(K)$ and two in $F$. Let $D$ be one of them. The two arcs on its boundary in $F$ cut a disk $D'$ from $F$ that has two arcs on its boundary also in $\partial N(K)$. Then $D \cup D'$ is either a properly embedded M\"obius band $Q$ or an annulus $A'$. We consider the annulus possibility first.  %Let $A'$ be the annulus obtained by taking the union of $D$ with $D'$.

If $A'$ is compressible, then we can use the compression disk together with half of $A'$ to obtain a disk with boundary consisting of two arcs, one in $F$ and one in $\partial N(K).$ But this contradicts the boundary-incompressibility of $F$.

If $A'$ is boundary compressible by a disk $D''$, we can take the arc in $D'' \cap A'$ to be in $D' \subset F$, therefore obtaining a boundary compression of $F$. So $A'$ is a essential annulus that does not intersect $F$.
Therefore the existence of $A'$ contradicts Lemma \ref{NoAnnulionOneSideSwitchMove}.

If $D \cup D'$ is a M\"obius band $Q$, then the boundary of $Q$ must be a meridian on $T_K$ as it is entirely contained in $M_1$ and cannot be trivial as then we would have a projective plane embedded in $M_1$ which we could embed in $S^3$, a contradiction.

The boundary of a regular neighborhood of $Q$ is an annulus $A''$. It is incompressible to the inside of the regular neighborhod of $Q$ as that is a solid torus, with the core curve of the annulus going around the core curve of the solid torus twice. It is incompressible to the outside as the boundaries are meridian curves on $T_K$.  It is boundary incompressible as any boundary compression would yield a boundary compression for $F$, a contradiction. So again, the existence of an essential annulus $A''$ that misses $F$ contradicts Lemma \ref{NoAnnulionOneSideSwitchMove}.%{\color{blue} Check added details here.}

%{\color{blue} Rewritten. Check the above.} {\color{red} Looks good, do we want to mention the Mobius Strip as a possibility?}{\color{blue} Good point! I have reworded.}

\medskip

Suppose $\partial_2 A$ is in some $T_{K_{1,i}}$. % The idea here is this implies K_2 is not linked with anything and cant pass through any holes, as this would kill the annulus 
%{\color{blue} Added following sentence.} {\color{red} Looks good.}
Then there must be an intersection arc in $A \cap F$ that cuts a disk from $A$ with one boundary in $F$ and the other boundary in $\partial N(K_2)$. We can use it to push $K_{2}$ onto $E$ by an isotopy in $H_L$ fixing the endpoints of $K_{2}$.  This implies that $H_L$ contains a compressing disk in $M_2$ with boundary isotopic in $\partial H$ to $\partial E$. We reach the analogous contradiction if  $\partial_2 A$ is in some $T_{K_{2,j}}$.

\vspace{0.5 cm}
\noindent \newline \textbf{Case 3:} The annulus $A$ has a boundary component $\partial_1 A$ on $\partial N(L)$ and a boundary component $\partial_2 A$ on $\partial H$. 
\vspace{0.5 cm}

Let $J$ be the component of $L$ with regular neighborhood boundary that $A$ intersects. Then the boundary of a regular neighborhood of $A \cup \partial N(J)$ is an annulus $A'$ with both of its boundaries in $\partial H$. The boundaries of $A'$ are two parallel nontrivial curves on the boundary of $H$ that are also parallel to the one boundary of $A$ on $\partial H$. Thus $A'$ must be incompressible. 

If $A'$ is boundary compressible, then do the boundary compression on the annulus $A'$ to obtain a disk $D''$  with boundary in $\partial H$. By boundary-irreducibility of $H_L$, $D''$ would have to have trivial boundary in $\partial H$. The boundary compression has the impact on $\partial A'$ of surgering the two curves along an arc running from one to the other.  Surgering two nontrivial parallel curves on a surface of genus at least two along an arc that is not in the annulus between the curves yields a nontrivial curve. So the boundary compression cannot be to that side. Thus the boundary compression must be to the side of the annulus in $\partial H$ shared by the two curves. But this side is a solid torus missing its core curve $J$, preventing a boundary compression to that side. So $A'$ is an essential annulus in $H_L$ with both boundaries on $\partial H$, contradicting Case 1.

\end{proof}

\begin{lemma}
The manifold $H_L$ contains no essential torus.
\label{NoToriSwitchMove}
\end{lemma}

\begin{proof}

Suppose $H_L$ contains an essential torus $\mathcal{T}$. We assume that $|\mathcal{T} \cap F|$ is minimal among all essential tori in $H_L$.

Suppose first that $\mathcal{T} \cap F = \emptyset$. Then $\mathcal{T} \subset M_1$ or $\mathcal{T} \subset M_2$. For convenience, we assume $\mathcal{T} \subset M_1$. Then we can view $\mathcal{T}$ as a torus $\overline{\mathcal{T}}$ in $H_1 \setminus \mathring{N}(L_{1})$ which we show is essential. 

Suppose $\overline{\mathcal{T}}$ is boundary parallel in $H_1 \setminus \mathring{N}(L_{1})$. Since $\partial H_1$ has genus at least $2$, $\overline{\mathcal{T}}$ must be parallel to a component of $\partial N(L_1)$. If it is boundary parallel to a component $J$, then $\overline{\mathcal{T}}$ must separate a solid torus from $H_1$ that has $J$ as its core curve. Since $F$ is to the side of $\overline{\mathcal{T}}$ that $H$ is, the solid torus cannot intersect $F_1$ either. So both the solid torus and $J$ are in $M_1$, and $\overline{T}$ is boundary parallel in $H_L$, contrary to our assumption.

Suppose $\overline{\mathcal{T}}$ is compressible in $H_1 \setminus \mathring{N}(L_{1})$.  Then a nontrivial curve $\gamma \subset \overline{\mathcal{T}}$ bounds a disk $D$ in $H_1 \setminus \mathring{N}(L_{1})$. We assume that $|D\cap F_1|$ is minimal among all compression disks of $\overline{\mathcal{T}}$ in $H_1 \setminus \mathring{N}(L_{1})$. Note that the components of $D \cap F$ are circles.  

If $D \cap F_1 = \emptyset$, then $D \subset M_1$, which implies that $\mathcal{T}$ is compressible in $H_L$, a contradiction. If $D \cap F_1 \neq \emptyset$, by incompressibility of $F$, an innermost circle of $D \cap F_1$ in $D$ is trivial in $F_1$, hence by irreducibility of $H_L$, we can reduce $|D \cap F_1|$ by an isotopy, contradicting minimality. It follows that $\overline{\mathcal{T}}$ is essential in $H_1 \setminus \mathring{N}(L_{1})$, which contradicts that $H_1 \setminus \mathring{N}(L_{1})$ is hyperbolic. Since $H_2 \setminus \mathring{N}(L_{2})$ is hyperbolic, we reach the analogous contradictions if $\mathcal{T} \subset M_2$.

\medskip

Suppose next that $\mathcal{T} \cap F \neq \emptyset$. By minimality of $|\mathcal{T} \cap F|$ and incompressibility of $F$, the components of $\mathcal{T} \cap F$ are circles which are nontrivial in $\mathcal{T}$ and $F$. 
%{\color{blue} There could be lots of circles, not just two, so adding a bit here.}

Let $A_C$ be an annulus which is a connected component of $M_1 \cap \mathcal{T}$ with boundary two circles in $F \cap \mathcal{T}$.
We claim the boundaries of $A_C$ are two disjoint circles $C_1$ and $C_2$ which bound disjoint disks in $E$ punctured once by $L$. Suppose otherwise. Then two circles $C_1,C_2 \subset A_C \cap F$ bound disks $D_1,D_2 \subset E$ such that $D_2 \subset D_1$. %We can assume that $C_1,C_2$ bound an annulus $A_C \subset \mathcal{T}$ such that $A_C \cap F = C_1 \cup C_2$. 
If $D_2$ is punctured once and $D_1$ is punctured twice by $L$, then we can glue $D_2$ and a slightly moved $D_1$ to $A_C$ to obtain a sphere in $H$ that is punctured three times by $L$.  Thus $D_1,D_2$ are both punctured once or twice by $L$.

 Suppose $D_1$ and $D_2$ are both punctured twice. Then by adding the annuli in $F \setminus D_i$ to $A_C$, we obtain an annulus $A'_C$ with boundary in $\partial H$. By the same reasoning as in the proof of Case $1$ in the proof of Lemma \ref{NoAnnuliSwitchMove}, $A'_C$ is boundary compressible in $M_1$ and we can push $A_C$ through $F$ to reduce $|A \cap F|$, contradicting minimality. 
 
 Suppose $D_1$ and $D_2$ are both punctured once. The circles $C_1$ and $C_2$  bound an annulus $A_{C,F}$ in $F$ which is not punctured by $L$. 
 
 By gluing the punctured disks $D_1$ and $D_2$ onto $A_C$, and sliding the $D_1$ portion just off $F$, we obtain a new annulus $\overline{A}_C$ with boundaries on $A_{K_1}$. 
 %{\color{blue} Lost here. I assume this does not mean to isotope $A_C$ onto $A_{K_1}$. Maybe only isotoping the boundary. Probably better just to glue the two disks onto $A_C$ and push one off the other to obtain the annulus you mean. Do I have this right?} 
 %{\color{red} Yeah that is correct, what I mean is to just slide the boundary components of $A_C$ down onto $A_{K_1}$.}{\color{blue}See rewrite.} %keeping the boundary components of $A_C$ on $\partial H_{1,1} \cup A_{K_1}$.
 This annulus $\overline{A}_C \subset M_1$ is properly embedded in $H_L$ with $\partial \overline{A}_C \subset T_{K}$. The  boundaries of $\overline{A}_C$ are meridians on $T_{K}$ that bound an annulus $A'_{C, F} \subset A_{K_1}$ which is obtained from $A_{C,F}$ by an isotopy in $M_1$. 
 Note $\overline{A}_C$ is incompressible in $H_L$ as $A_C$ is incompressible, and hence by Lemma \ref{NoAnnulionOneSideSwitchMove} it is boundary compressible in $H_L$. Thus a nontrivial arc $\alpha$ in $\overline{A}_C$ bounds a disk $D_{\beta}$ in $H_L$ with an arc $\beta \subset T_{K}$. 
 
 If $\beta$ is not a nontrivial arc in $A'_{C,F}$, it intersects $A_{K_2}$ in a nontrivial arc. In that case $D_{\beta}$ becomes a compressing disk for the torus $\overline{A}_C \cup (T_K \setminus A'_{C,F}$. Doin the compression yields a sphere in $H_L$ that separates $K$ from $\partial H$, a contradiction to irreducibility of $H_L$.
 %which implies that $H_L$ contains an essential sphere which bounds $3$-ball in $H$ containing $K$. {\color{blue} Not sure I see this. How are you getting a sphere? Taking regular neighborhood of $D_{\beta} \cup \overline{A}_C \cup N(K)$? Need more explanation here.} {\color{red} Yes, something like that should work. The picture I had in my head is that the existence of such a disk means $K$ can be separated from the other components by a sphere. Maybe we could instead say that doing the boundary compression for $\overline{A}_C$ in this case would yield a compression disk with boundary along $K$?} 
 
 If $\beta$ is not a nontrivial arc in $A'_{C,F}$, the disk $D_{\beta}$ lies in the region contained in $M_1$ that $\overline{A}_C$ separates from $H_L$. We can thus push $D_{\beta}$ by an isotopy to obtain a boundary compression disk for $A_C$ in $M_1$, hence $A_C$ is boundary compressible in $M_1$ and boundary parallel (since the boundary compressing arc in $M_1$ is a nontrivial arc in $A_{C,F}$) and we can push it through $F$ to reduce $|A \cap F|$, a contradiction. %{\color{blue} Check my rewrite above about the sphere. Sound okay?} {\color{red} Yep looks correct.}
 \newline
 
 We reach the analogous contradictions if $A_C \subset M_2$. Thus, we can assume the boundaries of $A_C$ are %components of $\mathcal{T} \cap F$ are 
 two disjoint circles which bound disjoint disks in $E$ punctured once by $L$. 

If there were more than one such annulus in $M_1$ and one such in $M_2$, then following along the annuli, one after the other as we travel along a longitude of $\mathcal{T}$, we would have to have them cycle one inside the next as they pass through $F$, and the torus could never close up.  So there is only one to each side of $F$ and $\mathcal{T}$ is cut into two incompressible (since the elements of $\mathcal{T} \cap F$ are nontrivial in $\mathcal{T}$) annuli $\mathcal{A}_1 \subset M_1, \mathcal{A}_2 \subset M_2$.

 If we glue the punctured disks $D_1$ and $D_2$ to $\mathcal{A}_1$ we obtain an incompressible annulus, which must then be boundary parallel to $\partial N(K)$ by Lemma \ref{NoAnnuliSwitchMove}. The same holds for $\mathcal{A}_2$, implying the torus $\mathcal{T}$ is boundary parallel, a contradiction to its being essential.

%{\color{blue} Now drop next paragraph. Let me know what you think.}{\color{red} Yep looks good! This is what I meant by $\mathcal{A}_1$ being boundary compressible.}

%If $\mathcal{A}_1$ is boundary compressible, {\color{blue} A little confused here. $\mathcal{A}_1$ is not a properly embedded annulus in $H_L$, only in $M_1$.}the boundary compressing arc in $\partial H_L$ must be contained in $M_1$,    {\color{blue} Same here. I don't see the sphere.} In particular, the boundary compressing arc is contained in the region in $M_1$ that $\mathcal{A}_1$ separates from $H_L$, hence $\mathcal{A}_1$ is boundary parallel. The same reasoning applies to $M_2$, hence if both of $\mathcal{A}_1,\mathcal{A}_2$ are boundary compressible, the torus $\mathcal{T}$ is boundary parallel. By Lemma \ref{NoAnnuliSwitchMove}, both $\mathcal{A}_1, \mathcal{A}_2$ are boundary compressible, thus we reach a contradiction since $\mathcal{T}$ is essential.
\end{proof}

\medskip
\medskip

A situation where Theorem \ref{SwitchMovetheorem} is easily applicable is when $H_1 = H_2,$ and $L_{1}= L_{2}$. See Figure \ref{CuttingOutPiece}.

\begin{figure}[htbp]
\includegraphics[scale=0.6]{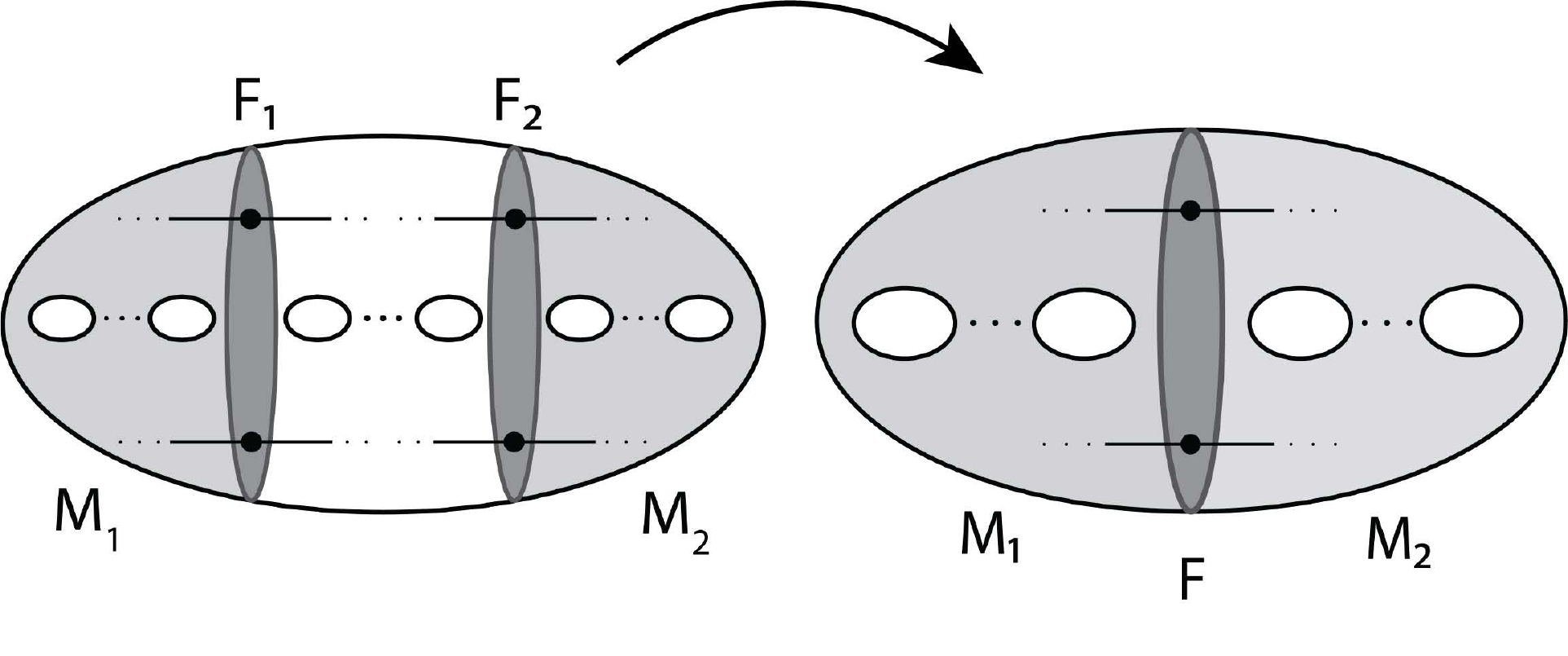}
\caption{Applying Theorem \ref{SwitchMovetheorem} to two pieces in a single handlebody.}
\label{CuttingOutPiece}
\end{figure}

\begin{corollary}
\label{CutOutPiece Corollary}
 Let $(H,L)$ be a handlebody/link pair that is tg-hyperbolic. Let $E_1$ and $E_2$ be two disjoint twice-punctured separating disks in $H$. Then cutting along the two disks, the piece with both disks on the boundary can be discarded and the two pieces with one disk along the boundary, assuming they are positive genus,  can be glued together along those disks, and the resulting handlebody/link pair will be tg-hyperbolic.
 \end{corollary}
 
%so that $M_1$ and $M_2$ are submanifolds of the same link complement and $E_1$ and $E_2$ are two disjoint twice-punctured separating disks in the same handlebody, each of which separates the manifold into handlebodies, such that the two pieces that intersect only one disk are positive genus. Theorem \ref{SwitchMovetheorem} then just requires that the complement of the original link in the original handlebody be tg-hyperbolic so that  part of the handlebody can be removed while still preserving tg-hyperbolicity. This is depicted in Figure \ref{CuttingOutPiece}.

Note that the intermediate piece that is being removed need not have positive genus. So, we can remove appropriate tangles from a tg-hyperbolic link in a handlebody and still preserve tg-hyperbolicity. Thus, in order to determine tg-hyperbolicity of a link in a handlebody, all such tangles could be removed and if the resulting simplified link is not tg-hyperbolic because of the presence of an essential sphere, disk, annulus or torus, neither could the original link have been.

The ideas in the proof of Theorem \ref{SwitchMovetheorem} extend to a different setting, where we cut a handlebody into three pieces along disks $E_1$ and $E_2$ and glue one piece to itself along the copies of $E_1$ and $E_2$. 

Suppose $L_{1}$ is a link in a handlebody $H_1$  and $(H_1,L_{1})$ is tg-hyperbolic. Suppose $E_1$ and $E_2$ are two nontrivial separating disks in $H_1$ each punctured twice by $L_1$, which together separate a handlebody $H_{1,2}$ of genus $g_{1,2}$ from two disjoint handlebodies $H_{1,1},H_{1,3}$ of genus $g_{1,1},g_{1,3}$ respectively, with all these genera positive. Let $M_{1,i} = H_{1,i}\setminus \mathring{N}(L_{1}),F_i = E_i \setminus \mathring{N}(L_{1})$. Let $L_{1,2} = L_{1} \cap H_{1,2}$.  
%Connecting the elements of $L_{1,2} \cap E_i$  with embedded arcs in $E_i$ respectively, and pushing the resulting curve into the interior of $H_{1,2}$ yields a link $\overline{L}_{1,2}$ in $H_{1,2}$. Gluing $M_{1,2}$ to itself by an orientation preserving homeomorphism $\phi: F_1 \to F_2$ sending $\partial E_1 $ to $\partial E_2$ and $\partial F_1 \cap \partial N(L_{1})$ to $\partial F_2 \cap \partial N(L_{2})$ yields a link complement $H_L = H \backslash \mathring{N}(L)$ in the handlebody $H$ of genus $g+1$ as in Figure \ref{SelfGluing}. We denote by $F$ the image of $F_1,F_2$ in $H_L$.

%{\color{blue} Older paragraph above removed and  rewritten.}{\color{red} Looks good, commenting old paragraph out.}

%Connecting the elements of $L_{1,2} \cap E_i$  with embedded arcs in $E_i$ respectively, and pushing the resulting curve into the interior of $H_{1,2}$ yields a link $\overline{L}_{1,2}$ in $H_{1,2}$. 

Gluing the subsets $F_1,F_2$ of $\partial M_{1,2}$ together by an orientation preserving homeomorphism $\phi: F_1 \to F_2$ sending $\partial E_1 $ to $\partial E_2$ and $\partial F_1 \cap \partial N(L_{1})$ to $\partial F_2 \cap \partial N(L_{2})$ yields a link complement $H_L = H \backslash \mathring{N}(L)$ in the handlebody $H$ of genus $g_{1,2}+1$ as in Figure \ref{SelfGluing}. We denote by $F$ the image of $F_1$ and $F_2$ in $H_L$.

\begin{figure}[htbp]
\includegraphics[scale=0.6]{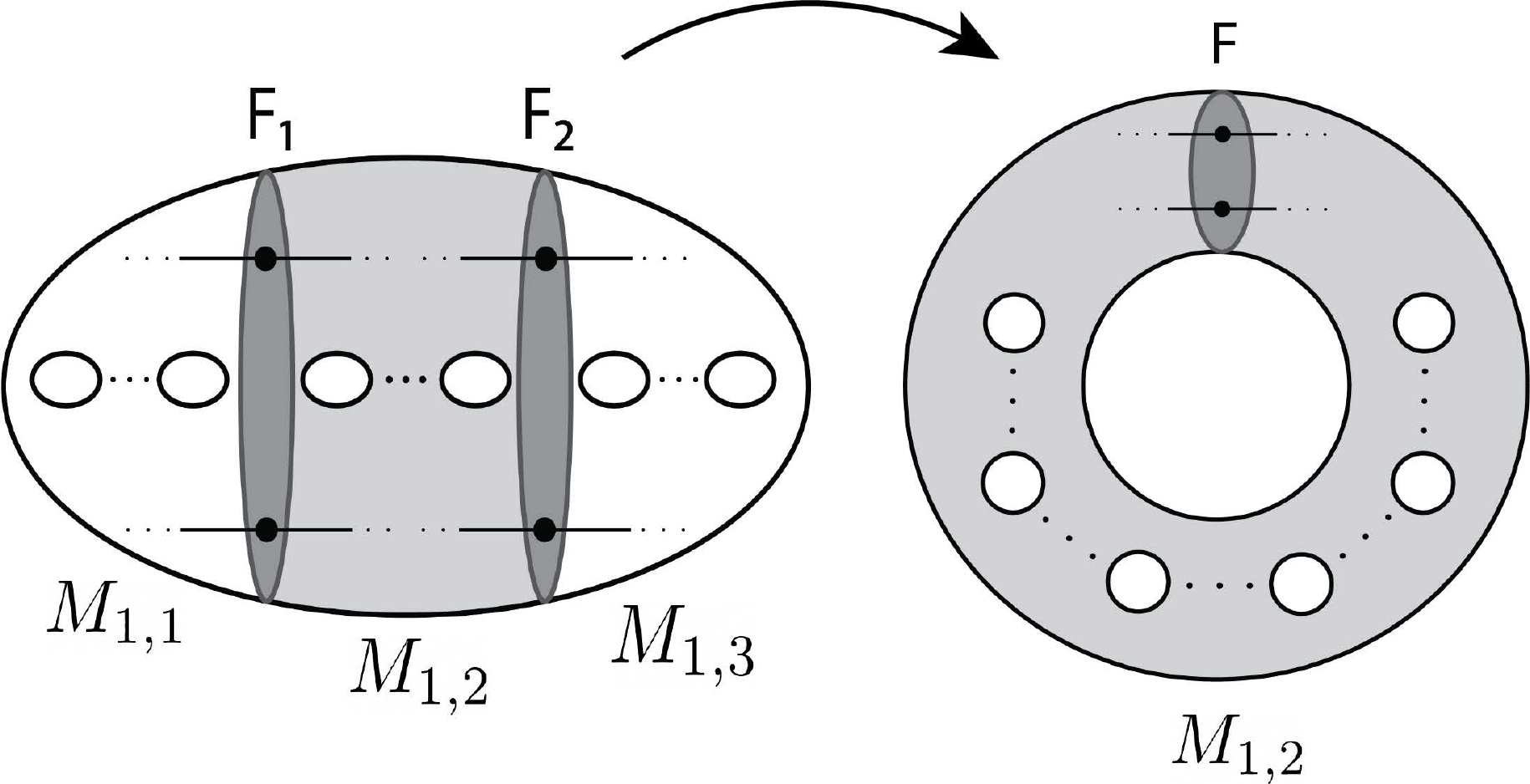}
\caption{Gluing $M$ to itself by a homeomorphism $F_1 \to F_2$.}
\label{SelfGluing}
\end{figure}

\begin{theorem}
\label{SelfGluingTheorem}

Suppose $H_{1} \setminus L_1$ is tg-hyperbolic, and  $E_1 \cap L_{1}=  E_1 \cap K, E_2 \cap L_{1} = E_2 \cap K'$, where $K$ and $K'$ are two distinct components of $L_{1}$,  then $H_L$ is tg-hyperbolic.
\end{theorem}  %{\color{blue} Do we need $H_{1,2} \setminus \mathring{N}(\overline{L}_{1,2})$  tg-hyperbolic any more?} {\color{red} I don't think we do.}
\noindent Theorem \ref{SelfGluingTheorem} follows from the same arguments as Theorem \ref{SwitchMovetheorem}. Namely, the surfaces $F,F_1,$ and $F_2$ are incompressible and boundary incompressible, and we can use this to reach the analogous contradictions from Lemmas \ref{SurfaceLemmaSwitchMove}-\ref{NoToriSwitchMove}. The requirement that the punctures of $E_1$ and $E_2$ correspond to two distinct components $K$ and $K'$  of $L$ must be introduced to force an annulus with boundary in $\partial H$ that intersects $F$ in nontrivial arcs to be cut into disks with two opposite sides in $F$. Without this condition the result does not hold in general, as shown in Figure \ref{CounterExample}.

\begin{figure}[htbp]
\includegraphics[scale=0.6]{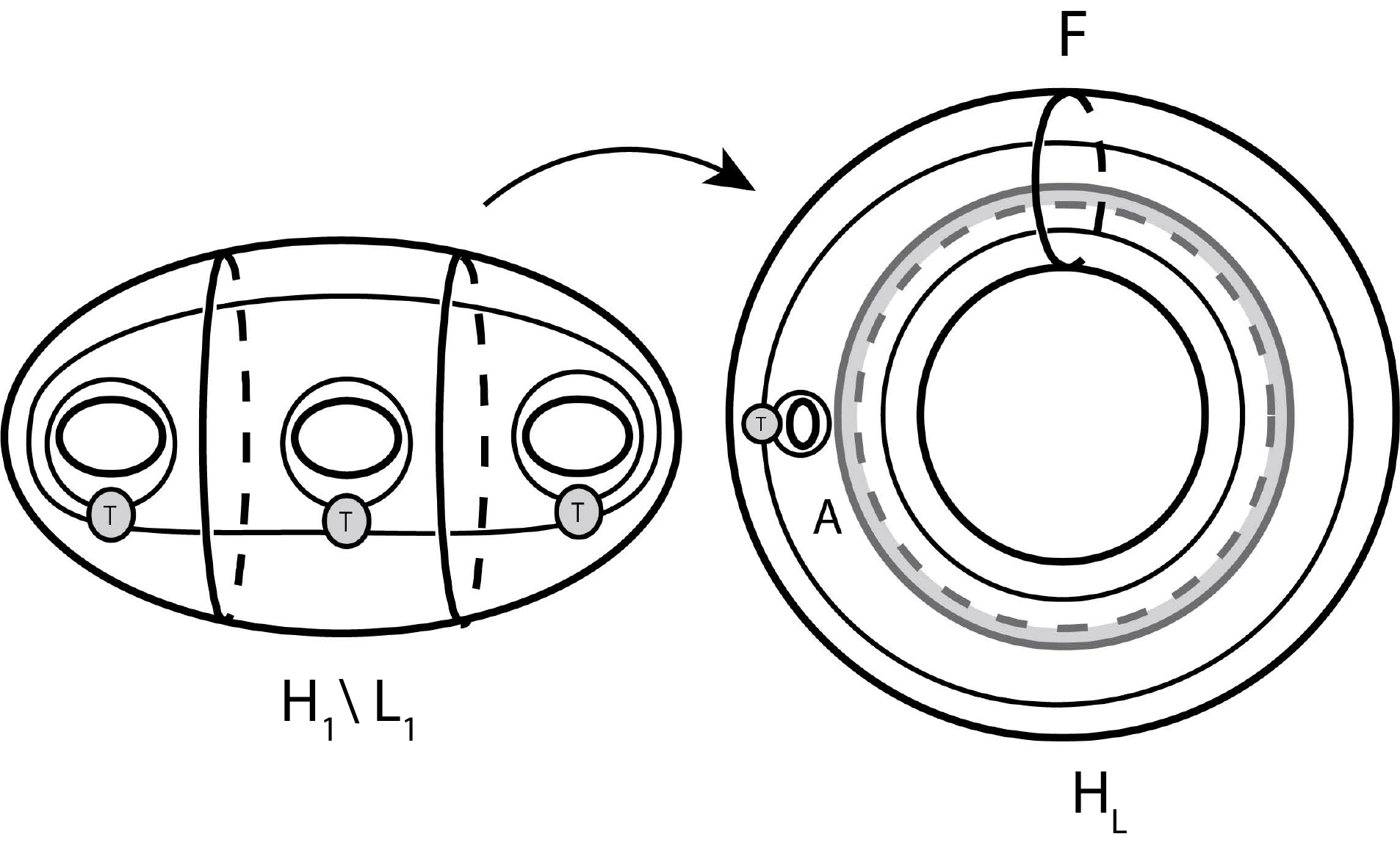}
\caption{A counterexample to Theorem \ref{SelfGluingTheorem} when the condition on the punctures of $E_1,E_2$ is removed. Here $T$ is an alternating tangle which can be chosen to satisfy the conditions of Theorem 1.6 of \cite{AltPaper} (appearing in the next section) so that $H_1 \setminus L_1$ tg-hyperbolic. After cutting and gluing, $H_L$ contains an essential annulus $A$ with boundary in $\partial H$ as shown (perpendicular to the page), which intersects $F$ in a single nontrivial arc and which separates one component of the link.%{\color{blue} Changed labels on link and modified caption.}
}
\label{CounterExample}
\end{figure}

\section{Applications}
\label{AppSection}

\subsection{Staked Links}

 Links in handlebodies are directly  related to the theory of staked links defined in \cite{generalizedknotoids}. (These links are also called tunnel links as in \cite{tunnellink} or starred links as in  as-of-yet unpublished work of N. G\"ug\"umc\"u and L. Kauffman.)  In this section we will only work with staked links in $S^2$. A \textit{staked link} is a pair $(L_D,\{p_i\}_{1 \leq i \leq n})$ of a link diagram $L_D \subset S^2$ together with a finite collection $\{p_i\}_{1 \leq i \leq n}$ of \textit{isolated poles}, which are distinct points $p_1,\dots,p_n \in S^2$ such that each $p_i$ lies in a connected component of $S^2 \setminus L_D$. Staked links are considered up to Reidemeister moves that do not pass strands over elements of $\{p_i\}_{1 \leq i \leq n}$. A staked link determines a link in a handlebody of genus $n-1$ as follows. Choose open disks $D_1,\dots,D_n \subset S^2 \setminus L_D$ containing $p_1, \dots, p_n$ respectively, such that $D_i \cap D_j = \emptyset$ for $i \neq j$. Then $D_L := S^2 \setminus (\cup_{i=1}^n D_i)$ is the closure of a $n-1$ punctured disk and $L_D$ determines a link $\overline {L}_D$ in the handlebody $D_L \times [0,1]$ as shown in Figure \ref{StakingFigure}. A staked link $(L_D,\{p_i\}_{1 \leq i \leq n})$ is tg-hyperbolic if $(D_L \times [0,1], \overline{L}_D)$ is hyperbolic as in Section \ref{Intro}.

 \begin{figure}[htbp]
 
 \includegraphics[scale=0.4]{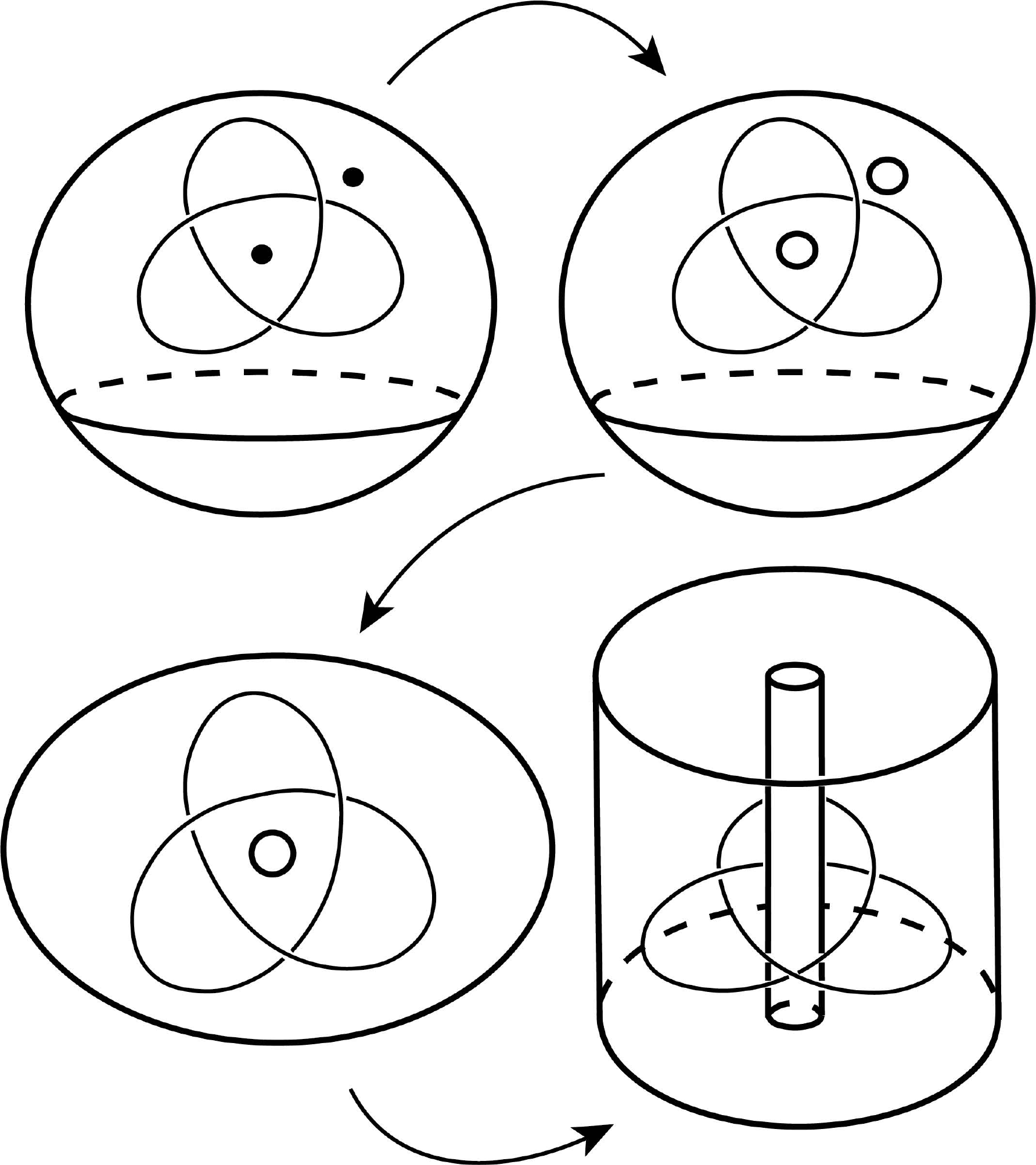}
 \caption{A staked link $L_D \subset S^2$ with $n$ stakes determines a link $\overline{L}_D$ in a handlebody of genus $n-1$.}
\label{StakingFigure} 
 \end{figure}
 
  \indent Given a staked link $(L_D,\{p_i\}_{1\leq i \leq n})$, any simple closed loop $\gamma:  [0,1] \to S^2$ with $\gamma(0) = \gamma(1) = p_i$ determines a proper non self-intersecting arc $a_{\gamma} \subset S^2 \setminus (\cup_{i=1}^n D_i)$ with $\partial a_{\gamma} \subset \partial D_i$, and hence a proper separating disk $a_{\gamma} \times [0,1]$ in $D_L \times [0,1]$, as in Figure \ref{StakedApplicationImage}. If $\gamma$ intersects $L_D$ twice, this disk could come from a gluing operation satisfying the conditions of Theorem \ref{SwitchMovetheorem}, hence Theorem \ref{SwitchMovetheorem} gives a way to check if a  complicated staked link is hyperbolic by checking if it is cut by $\gamma$ into pieces which come from hyperbolic staked links.

 \begin{figure}[htbp]
 
\includegraphics[scale=0.35]{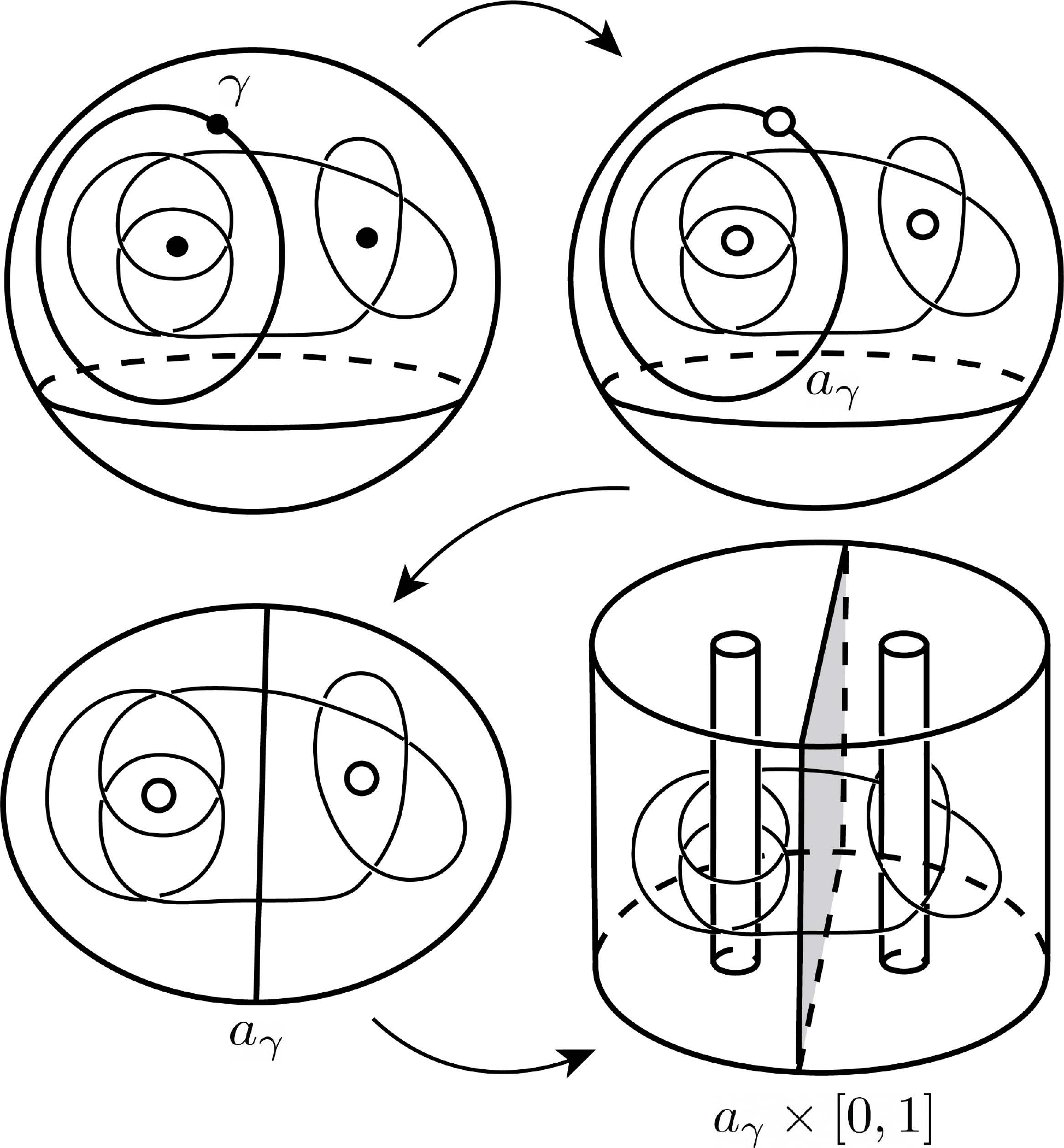}

\caption{A simple closed loop $\gamma$ based at a pole of a staked knot determines a separating disk in the corresponding handlebody.%{\color{blue} Daniel, can we modify the labels on this  picture a little? The $\gamma$ looks like a y. Try using font "symbol" (but maybe you already are.) And $a_\gamma$ in the second picture looks like it labels the pole rather than the arc. Maybe move the label down near the bottom of the arc. And for the last picture, switch label to be $\a_{\gamma} \times [0,1]$ to be consistent with text. I added a reference in the text to this figure.} {\color{red} Fixed,let me know} {\color{blue} Great! Thanks!}
}
\label{StakedApplicationImage} 
 \end{figure}

\subsection{Alternating Links}

To show a link in a handlebody $(H,L)$ is tg-hyperbolic, it is sufficient to show that $H$ can be given a product structure $H \cong F \times [0,1]$, where $F$ is the closure of a disk punctured some nonzero number of times, such that the projection of $L$ to the surface $F \times \{1/2\}$ is alternating and satisfies conditions as follows.

 \begin{theorem}[Theorem $1.6$ in \cite{AltPaper}] \label{thm-thickened}
Let $F$ be a projection surface with nonempty boundary which is not a disk, and let $L \subset F \times I$ be a link with a connected, reduced, alternating projection diagram $\pi(L) \subset F \times \{1/2\}$ with at least one crossing.  Let $M = (F \times I) \setminus N(L)$. Then $M$ is tg-hyperbolic if and only if the following four conditions are satisfied: 
\begin{enumerate} [label = (\roman*)]
    \item $\pi(L)$ is weakly prime on $F \times \{1/2\}$; 
    
    \item the interior of every complementary region of $(F \times \{1/2\}) \setminus \pi(L)$ is either an open disk or an open annulus;
    
    \item if regions $R_1$ and $R_2$ of $(F \times \{1/2\}) \setminus \pi(L)$ share an edge, then at least one is a disk;
    
    \item there is no simple closed curve $\alpha$ in $F$ that intersects $\pi(L)$ exactly in a nonempty collection of crossings,  such that for each such crossing, $\alpha$ bisects the crossing and the two opposite complementary regions meeting at that crossing that do not intersect $\alpha$ near that crossing are annuli.
\end{enumerate}

\end{theorem}

By weakly prime we mean that there is no simple closed curve on the projection surface that crosses the link twice and that bounds a disk that contains crossings. Note that each of these conditions is easily checked for the projection.

In the notations of Section \ref{SwitchSection}, this gives a simple way to show that $(H_1,L_{1})$ and $(H_2,L_{2})$ are tg-hyperbolic.  Note that Theorem \ref{SwitchMovetheorem} gives the expected behavior when both $L_{1},L_{2}$ are alternating and $K_1,K_2$ glue together so that $K$ is alternating. In particular, Theorem \ref{SwitchMovetheorem} can apply in the general situation of gluing an alternating piece to a non-alternating piece. 

As an example, for any weakly prime alternating tangle $T$ as in Figure \ref{AltApplication} other than 0 or 1 crossing or a horizontal sequence of bigons, (which do not satisfy the conditions of the theorem), we can form the piece $M_T$. Then if we take any other hyperbolic knot in a handlebody of positive genus, and split it into two pieces of positive genus by a twice-punctured disk, we can glue either resulting piece to the piece $M_T$ and still generate a tg-hyperbolic handlebody/link pair. % An example of this is shown in Figure \ref{AltApplication}.

\begin{figure}[htbp]
\includegraphics[scale=0.4]{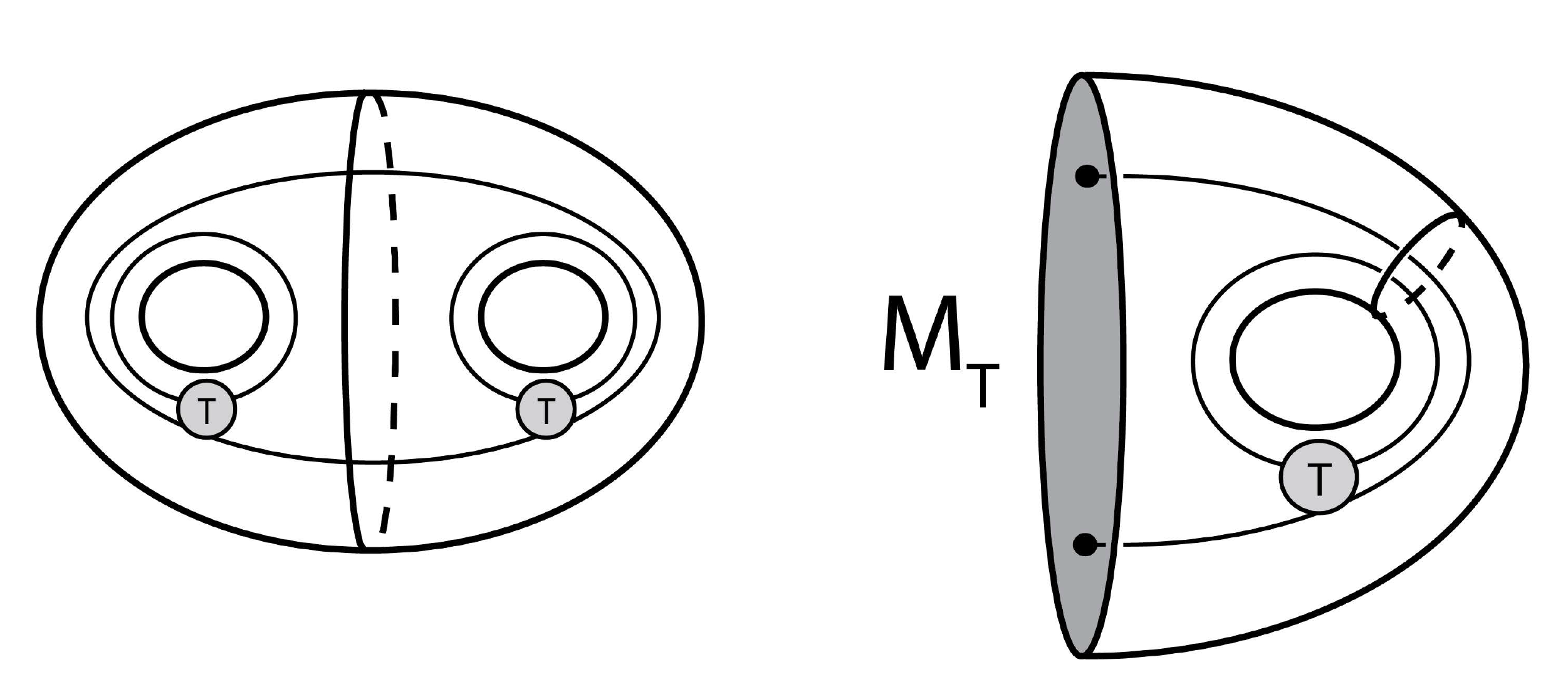}
\caption{If $T$ is an alternating tangle satisfying simple restrictions, the genus 2 handlebody/link pair depicted is tg-hyperbolic, so  we can glue $M_T$ to any other piece from a hyperbolic handlebody/link pair to obtain another tg-hyperbolic handlebody/link pair.  %Hence by Theorem \ref{SwitchMovetheorem} we can take any hyperbolic knot in a cut along any twice punctured nontrivial separating disk in a hyperbolic disk in  and glue in the manifold $M_T$ while still preserving hyperbolicity.{\color{blue} I am confused here. Why do we need genus 2 and genus 1 pictures?  And what does "nontrivial disk" mean in this context? That there is genus to both sides? If so, we need to make that explicit. Also, I would use an R for the tangle rather than a T, so it is clear whether or not the two tangles in the genus 2 case are the same as one another or reflections of one another.} 
} %{\color{blue} Note rewrite above and in caption. I believe for the example given, the conditions in the alternating theorem are satsified unless $T$ is a sequence of bigons. I also removed all references that were not cited. Turns out the nocite command does not work under our choice of bibliography system. }
\label{AltApplication}
\end{figure}

 \subsection{Planar Knotoids}

Knotoids are a variation on knots given by projections of line segments defined up to Reidemeister moves and disallowing strands to pass over or under the endpoints of the segment. When the projection surface is a plane, we say the knotoid is a planar knotoid. In \cite{hypknotoids}, two definitions of hyperbolicity of planar knotoids were given. The first, which is called the planar reflected doubling map, associates to the knotoid a link in a genus three handlebody. If the complement of the link is tg-hyperbolic, the knotoid is said to be hyperbolic under the reflected doubling map. The second, which is called the planar gluing map, associates to the knotoid a link in a genus two handlebody. Again, if the complement of the link is tg-hyperbolic, the knotoid is said to be hyperbolic under the gluing map. Proposition 2.5 in \cite{hypknotoids} proves that hyperbolicity of a planar knotoid under the reflected doubling map implies hyperbolicity under the gluing map but not vice versa. Further, the volume under the reflected doubling map is always at least as large as the volume under the gluing map. Theorem \ref{SwitchMovetheorem} together with the results from \cite{AltPaper} can provide many examples of planar knotoids that are hyperbolic  under either of the two constructions.
%{\color{blue} Above subsection is new.}

%{\color{red} In the notation of theorem \ref{SwitchMovetheorem}, the genus two picture can be taken for the pair $(H_1,L_1)$, and the genus one picture for $(H_{1,1},\overline{L}_{1,1})$, while $(H_2,L_2)$ is another arbitrary hyperbolic pair we are cutting a part out of. That is what I mean for nontrivial disk yes, I will edit it to make it more clear. I really mean the same tangle, we just need everything to be alternating to apply Joye's theorem.} {\color{red} With our new result we don't need the genus $1$ picture. I will edit the figure.}

\bibliographystyle{plain}
\bibliography{refs2}
\nocite{*}
\end{document}